\documentclass[12pt]{amsart}
\usepackage{amsmath}
\usepackage{amssymb}
\usepackage{xypic}
\usepackage{multicol}

\newcommand{\bA}{{\mathbb{A}}}
\newcommand{\bB}{{\mathbb{B}}}
\newcommand{\bG}{{\mathbb{G}}}
\newcommand{\bH}{{\mathbb{H}}}
\newcommand{\bP}{{\mathbb{P}}}
\newcommand{\bt}{{{\bf t}}}
\newcommand{\CC}{{\mathbb{C}}}

\newcommand{\PP}{{\mathbb{P}}}

\newcommand{\RR}{{\mathbb{R}}}
\newcommand{\ZZ}{{\mathbb{Z}}}

\newcommand{\calD}{{\mathcal{D}}}
\newcommand{\cE}{{\mathcal{E}}}
\newcommand{\cF}{{\mathcal{F}}}
\newcommand{\cG}{{\mathcal{G}}}

\newcommand{\cM}{{\mathcal{M}}}
\newcommand{\cP}{{\mathcal{P}}}
\newcommand{\cQ}{{\mathcal{Q}}}
\newcommand{\cU}{{\mathcal{U}}}
\newcommand{\cV}{{\mathcal{V}}}
\newcommand{\cX}{{\mathcal{X}}}

\newcommand{\cZ}{{\mathcal{Z}}}
\newcommand{\fM}{{\mathfrak{M}}}
\newcommand{\cO}{{\mathcal{O}}}
\newcommand{\fZrel}{{\mathfrak{Z}^{\rm rel}}}

\newcommand{\Spec}{{\rm Spec}}

\newcommand{\z}[2]{{z_{{#1}{#2}}}}
\newcommand{\w}[2]{{w_{{#1}{#2}}}}
\newcommand{\m}[2]{{\mu_{#1}^{(#2)}}}
\newcommand{\f}[2]{{{#1}_{#2}}}
\newcommand{\fn}[2]{{\tilde{#1}_{#2}}}
\newcommand{\fnl}[4]{{#1}_{#2, #3}^{(#4)}}

\theoremstyle{plain}
\newtheorem{theorem}{Theorem}[section]
\newtheorem{proposition}[theorem]{Proposition}
\newtheorem{corollary}[theorem]{Corollary}
\newtheorem{lemma}[theorem]{Lemma}
\newtheorem{claim}[theorem]{Claim}

\theoremstyle{definition}
\newtheorem{definition}[theorem]{Definition}
\newtheorem{example}[theorem]{Example}

\theoremstyle{remark}
\newtheorem{remark}[theorem]{Remark}

\author{Nobuyoshi Takahashi}
\address{Nobuyoshi Takahashi\\
Department of Mathematics\\
Graduate School of Science\\
Hiroshima University\\
1-3-1 Kagamiyama\\
Higashi-Hiroshima\\
739-8526 Japan}
\email{tkhsnbys@hiroshima-u.ac.jp}

\title{On the multiplicity of reducible relative stable morphisms}
\thanks{Partially supported by JSPS Grant-in-Aid for Scientific Research, No. 23540050}
\subjclass[2010]{14N35 (primary), 14D20 (secondary)}
\keywords{Enumerative geometry, Relative Gromov-Witten invariants, Plane curves.}

\begin{document}
\maketitle
\begin{abstract}
Let $(Z, D)$ be a pair of 
a smooth surface and a smooth anti-canonical divisor. 
Denote by $\fM_\beta$ 
the moduli stack of 
genus $0$ relative stable morphisms of class $\beta$ 
with full tangency to the boundary. 
Let $C_1$ and $C_2$ be rational curves 
fully tangent to $D$ at the same point $P$ and 
assume that $C_1$ and $C_2$ are immersed 
and that $(C_1.C_2)_P=\min\{D.C_1, D.C_2\}$. 
Then we show that the contribution of $C_1\cup C_2$ 
to the virtual count of $\fM_{[C_1]+[C_2]}$ 
is $\min\{D.C_1, D.C_2\}$. 

As an example, 
we describe genus $0$ relative stable morphisms 
to $(\PP^2, (\hbox{cubic}))$ of degree $4$ with full tangency, 
and examine how they contribute to 
the relative Gromov-Witten invariant. 
\end{abstract}

\section{Introduction}

In most of recent study of enumerative geometry, 
one mainly deals with invariants 
defined from certain compactified moduli spaces 
rather than naive counting invariants, 
since the former is better-behaved in theoretical frameworks. 
A typical example is that of Gromov-Witten invariants, 
which is defined from moduli spaces of stable morphisms. 

To go from the naive counting to the moduli theoretic invariants, 
or in the other direction, 
one has to calculate the contribution from the degenerate objects. 
In the case of Gromov-Witten invariants of a Calabi-Yau $3$-fold $X$, 
multiple covers must be taken into account. 
For example, 
if $C$ is an infinitesimally rigid smooth rational curve in $X$, 
then the $d$-fold covers of $C$ contribute $1/d^3$ 
to the genus $0$ Gromov-Witten invariant of class $d[C]$ (\cite{AM}, \cite{M}). 

There is a variant of the Gromov-Witten invariant, 
called the relative Gromov-Witten invariant(\cite{IP}, \cite{LR}, \cite{L2}), 
which corresponds to enumeration of curves on an open variety, 
or curves on a projective variety 
satisfying certain tangency conditions to a `boundary' divisor. 
An interesting case is that of a log Calabi-Yau surface, 
i.e. a pair $(Z, D)$ of a smooth projective surface $Z$ 
and a normal crossing divisor $D$ such that $K_Z+D\sim 0$. 
In this case, the first type of curves to consider are 
those curves $C$ such that the normalization of $C\setminus D$ 
is isomorphic to $\bA^1$. 
Let us call them rational curves with full tangency to $D$, 
and the corresponding moduli theoretic invariant 
the genus $0$, full tangency relative Gromov-Witten invariant. 
The expected dimension of the moduli space of such curves is $0$, 
and so one would like to know what the relative Gromov-Witten invariants are. 
In the case of $(\bP^2, (\hbox{cubic}))$, 
all the genus $0$, full tangency 
relative Gromov-Witten invariants were calculated in \cite{G}. 
In \cite{GPS}, 
the relative Gromov-Witten invariants of 
blown-up toric surfaces 
were related to the tropical vertex. 

Concerning the relation with the naive counting, 
the multiple cover formula 
was worked out in \cite[Proposition 6.1]{GPS}. 
\begin{proposition}(\cite[Proposition 6.1]{GPS})
Let $C$ be an immersed rational curve with full tangency to $D$ 
and write $w=(C.D)$. 
Then the contribution of $d$-fold covers of $C$ 
to the genus $0$, full tangency relative Gromov-Witten invariant 
of class $d[C]$ is 
\[
M_w[d]=\frac{1}{d^2}
\left(
\begin{array}{c}
d(w-1)-1 \\
d-1
\end{array}
\right). 
\]
\end{proposition}

For the relative Gromov-Witten invariants of log Calabi-Yau surfaces, 
we often have another kind of degenerations, 
namely reducible curves. 

\begin{example}
Let $E$ be a smooth cubic curve, 
and let us consider the relative Gromov-Witten invariants of $(\bP^2, E)$. 
Let $P\in E$ be a flex. 
Then the tangent line $L$ to $E$ at $P$ is a rational curve 
with full tangency to $E$. 
If $E$ is general, 
we can show that there are two nodal cubic curves $C_1, C_2$ 
with full tangency to $E$ at $P$. 
Then, $L\cup C_1$ and $L\cup C_2$ can contribute 
to the genus $0$, full tangency relative Gromov-Witten invariant. 
\end{example}

\begin{example}
Let us look at an example related to \cite[Examples 6.4]{GPS}. 
Let $D_1, D_2$ and $D_{out}$ be the coordinate lines of $\bP^2$, 
and take general points $x_{i1}, x_{i2}, x_{i3}\in D_i$ ($i=1, 2$). 
Let ${\bf P}=({\bf P}_1, {\bf P}_2)$ be a pair of ordered partitions 
of a positive integer $p_{out}$ into $3$ parts, i.e. 
${\bf P}_i$ is a triplet of nonnegative integers $p_{i1}, p_{i2}, p_{i3}$ 
with $\sum_{j=1}^3 p_{ij}=p_{out}$. 
We write ${\bf P}_i=p_{i1}+p_{i2}+p_{i3}$. 
Let $X^\circ[{\bf P}]$ denote the blow-up of $\bP^2\setminus\{(1:0:0), (0:1:0), (0:0:1)\}$ 
at $x_{11}, \dots, x_{23}$ 
(and so, it is in fact independent of ${\bf P}$), 
$D_{out}^\circ$ the proper transform of $D_{out}$, 
$H$ the pullback of a line 
and $E_{ij}$ the exceptional curve over $x_{ij}$. 

Consider the moduli space $\bar{\fM}(X^\circ[{\bf P}]/D_{out}^\circ)$ 
of genus $0$ relative stable morphisms 
of class $p_{out}H-\sum p_{ij}E_{ij}$ 
with full tangency to $D_{out}^\circ$. 
Although $X$ is not compact, this moduli space is compact. 
Its virtual Euler characteristic is a kind of relative Gromov-Witten invariant, 
and we denote it by $N_{\bf m}[{\bf P}_1, {\bf P}_2]$, 
where ${\bf m}=((1, 0), (0, 1))$ corresponds to 
the base surface $\PP^2$. 

The calculation in \cite[Examples 6.4]{GPS} gives 
$N_{\bf m}[1+1+1, 1+1+1]=18$. 
There are three points $P_1$, $P_2$, $P_3$ on $D_{out}$ 
at which a cubic through $x_{ij}$ can be fully tangent to $D_{out}$. 
By considering the monodromy, 
we see that for each $k$ there are $6$ rational cubics 
corresponding to elements of $\bar{\fM}(X^\circ[1+1+1, 1+1+1]/D_{out}^\circ)$ 
that touch $D_{out}$ at $P_k$. 

Now, if we take such curves $C_1$ and $C_2$ through $P_1$, 
then $C_1\cup C_2$ can contribute to 
$N_{\bf m}[2+2+2, 2+2+2]$. 
\end{example}

Thus we need to know the contributions from relative stable morphisms 
with reducible image curves. 
The following is our main theorem. 

\begin{theorem}\label{theorem_main}
Let $(Z, D)$ be a pair of 
a smooth surface and a smooth divisor. 
Denote by $\fM_\beta$ 
the moduli stack of 
genus $0$ relative stable morphisms of class $\beta$ 
with full tangency to $D$. 

Let $C_1$ and $C_2$ be curves satisfying the following: 
\begin{enumerate}
\item 
$C_i$ is a rational curve of class $\beta_i$ with full tangency to $D$, 
\item
$(K_Z+D).\beta_i=0$, 
\item
$C_1\cap D$ and $C_2\cap D$ consist of the same point $P$, 
and 
\item
$f_i$ are immersive and 
$(C_1.C_2)_P=\min\{D.C_1, D.C_2\}$. 
\end{enumerate}

Then there exists a unique point $[f]\in\fM_{\beta_1+\beta_2}$ 
whose image in $Z$ is $C_1\cup C_2$, 
and it is isolated with multiplicity $\min\{D.C_1, D.C_2\}$. 
\end{theorem}

The plan of this paper is as follows. 
In \S2, we recall the definitions and facts related to relative stable morphisms from \cite{L1}. 
The main theorem is proved in \S3. 
In \S4, 
we calculate the number of rational quartics with full tangency 
to a smooth cubic, 
and confirm that the multiple cover formula 
and our formula 
give the correct relative Gromov-Witten invariant.

\subsection{Acknowledgment}
The author thanks Rahul Pandharipande 
for discussions on relative Gromov-Witten invariants 
which motivated this work.

\section{Relative stable morphisms}

In this section, 
we recall some definitions and facts about relative stable morphisms 
from \cite{L1}. 

Let $Z$ be a smooth projective variety 
and $D$ a smooth connected divisor in $Z$. 

\subsection{Standard models of expanded relative pairs, 
the stack of expanded relative pairs(\cite[\S 4.1]{L1})}

One notable thing about 
the theory of relative stable morphisms is 
that it deals with morphisms to degenerate target spaces. 
For that, one first introduces certain degenerating families, 
called standard models of expanded relative pairs. 
Here we only explain what will be needed later. 

For each nonnegative integer $n$, 
the standard model $(Z[n], D[n])$ is a pair 
of a smooth variety $Z[n]$ equipped with morphisms to $\bA^n$ and $Z$
and a smooth divisor $D[n]$ on $Z[n]$. 
Via the morphism to $\bA^n$, 
it is regarded as a family of relative pairs over $\bA^n$. 

For $n=0$, 
let $(Z[0], D[0])=(Z, D)$. 

For $n=1$, 
let $Z[1]$ be the blow-up of $Z\times\bA^1$ 
with center $D\times\{0\}$, 
and $D[1]$ the proper transform of $D\times\bA^1$. 
We write $t$ for the coordinate of $\bA^1$. 
The fiber $(Z[1]_0, D[1]_0)$ over $(t=0)\in\bA^1$ can be described as follows. 
Let $\Delta=\PP(\cO_D\oplus N_{D/Z})$ 
be the $\bP^1$-bundle over $D$ 
associated to $\cO_D\oplus N_{D/Z}$ in the covariant way. 
Let $D_\infty$ and $D_0$ be sections corresponding to 
$\cO_D\oplus 0$ and $0\oplus N_{D/Z}$. 
Then $Z[1]_0$ is obtained by 
gluing $\Delta$ and $Z$ along $D_0$ and $D$, 
and $D[1]_0$ corresponds to $D_\infty$. 

Consider the morphism $\bA^n\to\bA^1$ 
given by $(t_1, \dots, t_n)\mapsto t_1\cdots t_n$. 
Then $Z[n]$ is obtained as a certain birational modification 
of $Z[1]\times_{\bA^1}\bA^n$. 
We skip the precise definition, 
and only state the following properties. 
\begin{enumerate}
\item
$Z[n]$ is equipped with a proper birational morphism $Z[n]\to Z\times\bA^n$, 
and the family $Z[n]\to\bA^n$ is flat and proper. 
We denote by $\varphi: Z[n]\to Z$ the projection to $Z$. 
\item
Let $G=\bG_m^n$ act on $\bA^n$ in the natural way. 
Then the induced $G$-action on $Z\times\bA^n$ 
lifts to $Z[n]$. 
\item
Let $\bH_l:=\{(t_1, \dots, t_n)\in\bA^n| t_l=0\}$. 
Then the non-smooth locus of $Z[n]\to\bA^n$ is 
$\coprod_{l=1}^n \bB_l\subset Z[n]$, 
where $\bB_l$ is mapped isomorphically 
to $D\times\bH_l\subset Z\times\bA^n$. 
\item
If $k$ components of $\bt=(t_1, \dots, t_n)$, 
say $t_{l_1}, \dots, t_{l_k}$, are zero 
and the others are nonzero, 
then the fiber $Z[n]_{\bt}$ over $\bt$ can be written as 
$\Delta_1\cup\dots\cup\Delta_k\cup\Delta_{k+1}$, 
where 
\begin{itemize}
\item
$\Delta_1, \dots, \Delta_k$ 
are isomorphic to $\Delta$, 
$\Delta_{k+1}$ is isomorphic to $Z$ via $\varphi$, 
\item
$\Delta_i\cap\Delta_{i+1}=(\bB_{l_i})_{\bt}$, 
and it corresponds to 
$D_0$ (resp. $D_{\infty}$ or $D$) 
under the isomorphism of $\Delta_i$ and $\Delta$
(resp. $\Delta_{i+1}$ and $\Delta$ or $Z$). 
\item
$\Delta_i\cap\Delta_j=\emptyset$ 
if $|i-j|>1$, and 
\item
$D[n]_{\bt}$ is contained in $\Delta_1$, 
and it corresponds to $D_\infty$ (or $D$) 
under the isomorphism of $\Delta_1$ and $\Delta$ (or $Z$). 
\end{itemize}
\item
For any $P\in(\bB_{l_i})_{\bt}$, 
there are regular functions $w_1$ and $w_2$ on a neighborhood of $P$ 
with the following property. 
If $w'_1, \dots, w'_{\dim Z-1}$ are functions on a neighborhood of $\varphi(P)$ 
in $Z$ 
that restrict to \'etale coordinates on $D$, 
then $Z[n]$ is \'etale locally isomorphic to 
\[
\{(w_1, w_2, w'_1, \dots, w'_{\dim Z-1}, t_1, \dots, t_n)| w_1w_2=t_{l_i}\}, 
\]
and $\Delta_i$(resp. $\Delta_{i+1}$) is defined by $w_1$(resp. $w_2$). 
\end{enumerate}

\begin{definition}
(1)
Let $S$ be a scheme. 
A family of expanded relative pairs of $(Z, D)$ over $S$ is 
a pair $(\cZ, \calD)$ equipped with a morphism $\cZ\to Z\times S$, 
such that the following holds: 
$\cZ$ is a scheme, $\calD$ is a closed subscheme of $\cZ$, 
and there exist an open covering $S=\bigcup S_\alpha$, 
natural numbers $n_\alpha$ and morphisms $S_\alpha\to\bA^{n_\alpha}$ 
such that $(\cZ, \calD)\times_S S_\alpha$ 
is isomorphic to $(Z[n_\alpha], D[n_\alpha])\times_{\bA^{n_\alpha}} S_\alpha$ 
over $Z\times S$. 

When two families $(\cZ_1, \calD_1)$  and $(\cZ_2, \calD_2)$ 
over the same base $S$ are given, 
an isomorphism of these families 
means an isomorphism over $Z\times S$. 

If $\xi=(\cZ, \calD)$ is a family over $S$ 
and $\rho: T\to S$ is a morphism, 
the pullback $\rho^*\xi$ is defined as 
the family $(\cZ\times_S T, \calD\times_S T)$ over $T$. 

(2)
Let $\fZrel$ be the category defined as follows. 
An object of $\fZrel$ is a family of expanded relative pairs of $(Z, D)$ over 
some scheme $S$. 
If $\xi_i$ ($i=1, 2$) are families over $S_i$, 
a morphism $\xi_1\to \xi_2$ 
is a pair of a morphism $\rho: S_1\to S_2$ 
and an isomorphism $\xi_1\cong\rho^*\xi_2$. 
\end{definition}

\begin{proposition}
Let $\mathfrak{p}: \fZrel\to\mathfrak{Sch}$ be 
the functor which sends a family over $S$ to $S$. 
Then $(\fZrel, \mathfrak{p})$ is a stack. 
\end{proposition}

\subsection{Relative stable morphisms and their moduli(\cite[\S 4.2]{L1})}

\begin{definition}\label{def_graph}
An admissible weighted graph $\Gamma$ is 
a finite graph with vertices $V(\Gamma)$ 
and no internal edges, 
along with
\begin{itemize}
\item
an ordered collection of external edges, called legs, 
\item
an ordered collection of weighted external edges, called roots, and 
\item
two weight functions $g: V(\Gamma)\to\ZZ_{\geq 0}$ (genera) 
and $b: V(\Gamma)\to A_1(Z)/\sim_{\rm alg}$ (degrees), 
\end{itemize}
such that the graph is connected 
when all roots are considered as connected. 

An isomorphism of graphs is a bijection between the sets of vertices, 
legs and roots 
preserving adjacency, weight functions on vertices, 
the orders of legs and roots, 
and weights of roots. 
\end{definition}

Let $l$, $k$ and $r$ denote the number of vertices, 
legs and roots 
and write $v_1, \dots, v_l$ for the vertices 
and $\mu_1, \dots, \mu_r$ for the weights of the roots.

\begin{remark}
In this paper, 
we will only use the case with one vertex of genus $0$ and one root. 
\end{remark}

\begin{definition}
Let $S$ be a scheme 
and $\Gamma$ an admissible graph. 
An $S$-family of relative stable morphisms to $(Z[n], D[n])$ 
of type $\Gamma$ is 
a quadruple $(f, \cX, (\cQ_i)_{i=1}^r, (\cP_j)_{j=1}^k)$, 
where 
\begin{enumerate}
\item
$\cX=\cX_1\coprod\dots\coprod\cX_l$, 
where $\cX_h$ is a family of (connected) prestable curves of genus $g(v_h)$ over $S$. 
(Note that the ordering of $\cX_h$ is determined by the next condition 
since either $l=1$ 
or each vertex has at least one root attached.) 
\item
$\cQ_i: S\to \cX$ and $\cP_j: S\to\cX$ are disjoint sections 
to the smooth locus of $\cX\to S$ 
such that $\cQ_i(S)\subset\cX_h$(resp. $\cP_j(S)\subset\cX_h$) 
if the $i$-th root(resp. $j$-th leg) is 
attached to $v_h$. 
\item
$f: \cX\to Z[n]$ is an $\bA^n$-morphism for a (unique) morphism $S\to\bA^n$, 
$f^{-1}D[n]=\sum_{i=1}^r \mu_i \cQ_i(S)$ 
and for each closed point $s\in S$ the class $(\varphi\circ f)_*((\cX_h)_s)$ 
in $Z$ is $b(v_h)$. 
\item
The morphism $f|_{\cX_h}$ 
together with marked sections in $\cX_h$ 
is a family of stable morphisms to $Z[n]$. 
\end{enumerate}

We call $\cQ_i$ the distinguished marked sections 
and $\cP_j$ the ordinary marked sections. 
\end{definition}

\begin{definition}
(1)
(\cite[\S 2.1]{L1})
Let $k[t]\to k[w_1, w_2]$ be the homomorphism 
defined by $t\mapsto w_1w_2$, 
$\phi: k[[s]]\to k[[z_1, z_2]]$ the continuous homomorphism 
defined by $s\mapsto z_1z_2$. 
For a $k[[s]]$-algebra $A$ which is $(s)$-adically complete, 
let $R=k[[z_1, z_2]]\otimes_{k[[s]]}A$ 
and write $\hat{R}$ for the $(z_1, z_2)$-adic completion of $R$. 

Let $\psi: k[t]\to A$ be a homomorphism 
and regard $\hat{R}$ as a $k[t]$-algebra by the induced homomorphism. 
Then a $k[t]$-homomorphisms $\varphi: k[w_1, w_2]\to\hat{R}$ 
is said to be of pure contact of order $n$ if, 
possibly after exchanging $z_1$ and $z_2$, 
one has 
\[
\varphi(w_1)=z_1^n\beta_1, \varphi(w_2)=z_2^n\beta_2 
\]
for units $\beta_1, \beta_2\in \hat{R}^*$ satisfying $\beta_1\beta_2\in A^*$. 

(2)
(\cite[\S 2.2]{L1})
Let $S$ be scheme over $\bA^n$, 
$\cX$ a disjoint union of families of prestable curves over $S$ 
and $f: \cX\to Z[n]$ a morphism over $\bA^n$. 
Then $f$ is called predeformable if the following hold. 

For any $l\in[1, n]$ and 
$R\in f^{-1}(\bB_l)$, 
let $b\in S$ and $\bt\in\bA^n$ be points below $R$ 
and write $P=f(R)$. 
Let $w_1$ and $w_2$ be formal functions near $P$ 
as in the description of $Z[n]$. 
The first requirement is that $R$ is a node in $\cX_b$. 
Choose formal functions $s$ on $S$ 
and $z_1, z_2$ on $\cX$ 
such that $(k[[z_1, z_2]]\otimes_{k[[s]]}\hat{\cO}_{S, b})^\wedge$
is isomorphic to $\hat{\cO}_{\cX, R}$. 
Write $t=t_l$. 
Then we require that the induced homomorphism 
$k[w_1, w_2]\to (k[[z_1, z_2]]\otimes_{k[[s]]}\hat{\cO}_{S, b})^\wedge$ 
is of pure contact. 
\end{definition}

\begin{remark}
(1)
The definition of predeformablity does not depend on 
the choice of $s, z_1, z_2, w_1$ and  $w_2$. 
If they are regular functions, 
one may take $\beta_1, \beta_2\in \cO_{\cX, R}^*$ 
with $\beta_1\beta_2\in\cO_{S, b}^*$. 

(2)
If $f$ is predeformable, 
then it is nondegenerate, 
i.e. for any $s\in S$, 
no irreducible component of $\cX_s$ is mapped into any of $\bB_i$. 

(3)
If $S=\Spec \CC$, 
the condition is as follows. 
If $f(R)\in\bB_i=\Delta_i\cap \Delta_{i+1}$, 
then $R$ is a node, 
the two branches at $R$ map to $\Delta_i$ and $\Delta_{i+1}$, 
and the pullbacks of $\bB_i$ 
to the two branches have the same order at $R$. 
\end{remark}

\begin{definition}
(1)(\cite[Definition 4.8]{L1}) 
An $S$-family of relative prestable morphisms to $\fZrel$ of type $\Gamma$ 
is a data $\xi=(f, \cX, (\cQ_i)_{i=1}^r, (\cP_j)_{j=1}^k, \cZ, \calD)$ where 
\begin{itemize}
\item
$\cX$ is flat family over $S$, 
\item
$(\cZ, \calD)$ is an object of $\fZrel(S)$ and 
\item
$f: \cX\to\cZ$ is a morphism over $S$, 
\end{itemize}
such that 
there exist an open covering $S=\bigcup S_\alpha$, 
natural numbers $n_\alpha$ and morphisms $S_\alpha\to\bA^{n_\alpha}$ 
satisfying the following properties: 
\begin{itemize}
\item 
$(\cZ, \calD)\times_S S_\alpha$ 
is isomorphic to $(Z[n_\alpha], D[n_\alpha])\times_{\bA^{n_\alpha}} S_\alpha$ 
over $Z\times S$. 
\item
The induced quadruple 
$(f, \cX, (\cQ_i), (\cP_j))\times_S S_\alpha$ 
is an $S_\alpha$-family of 
relative stable morphisms to $(Z[n_\alpha], D[n_\alpha])$ 
of type $\Gamma$. 
\end{itemize}

(2)
Let 
$\xi'=(f', \cX', (\cQ'_ i)_{i=1}^r, (\cP'_j)_{j=1}^k, \cZ', \calD')$ 
be another 
$S$-family of relative prestable morphisms to $\fZrel$ of type $\Gamma$. 
An isomorphism from $\xi$ to $\xi'$ is a pair $(r_1, r_2)$ 
consisting of 
\begin{itemize}
\item
an isomorphism $r_1: \cX\to\cX'$ over $S$ 
compatible with marked sections, and 
\item
an isomorphism $r_2: \cZ\to\cZ'$ over $Z\times S$
\end{itemize}
such that $r_2\circ f=f\circ r_1$. 

(3)
The family $\xi$ is called stable if 
its local representatives 
$(f, \cX, (\cQ_i), (\cP_j))\times_S S_\alpha$ 
are predeformable and 
the group of automorphisms of 
$\xi_b$ is finite 
for any closed point $b\in S$. 
\end{definition}

\begin{definition}(\cite[Definition 4.9]{L1}) 
Let $\fM(\fZrel, \Gamma)$ denote the category of 
relative stable morphisms to $\fZrel$ of type $\Gamma$. 
\end{definition}

\begin{theorem}(\cite[Theorem 4.10]{L1})
Let $\mathfrak{p}: \fM(\fZrel, \Gamma)\to\mathfrak{Sch}$ be 
the functor which sends families over $S$ to $S$. 
Then $(\fM(\fZrel, \Gamma), \mathfrak{p})$ is an algebraic stack, 
separated and proper over $k$. 
\end{theorem}

\begin{proposition}\label{prop_quotient}
For an admissible weighted graph $\Gamma$ 
as in Definition \ref{def_graph}, 
denote by $\fM(Z[n], \Gamma)$ the moduli space of 
stable morphisms $(f, X, (Q_i)_{i=1}^r, (P_j)_{j=1}^k)$, where 
\begin{itemize}
\item
there is given a one-to-one correspondence $v\mapsto X_v$ 
between vertices of $\Gamma$ and connected components of $X$, 
\item
$Q_i$ (resp. $P_j$) is a smooth point of $X$ 
which is contained in $X_v$ if $i$-th root (resp. $j$-th leg) 
is connected to $v$, 
\item
$p_a(X_v)=g(v)$, and 
\item
$f: X\to Z[n]$ is a stable morphism such that 
${\rm Im}(f)$ is contained in a fiber of $Z[n]\to\bA^n$ 
and $(\varphi\circ f)_*[X_v]=b(v)$. 
\end{itemize}
Then it has a locally closed substack $\fM((Z[n], D[n]), \Gamma)^{st}$ 
consisiting of relative stable morphisms to $\fZrel$. 
The $\bG_m^n$-action on $Z[n]$ induces 
a $\bG_m^n$-action on $\fM((Z[n], D[n]), \Gamma)^{st}$, 
and the natural morphism $\fM((Z[n], D[n]), \Gamma)^{st}/\bG_m^n\to \fM(\fZrel, \Gamma)$ 
is \'etale. 
\end{proposition}
\begin{proof}
Actually, this is proved in the course of the proof of the previous theorem. 
See the proof of \cite[Theorems 2.11 and 3.10]{L1}. 
\end{proof}

\section{Proof of the Main theorem}

Let us prove Theorem \ref{theorem_main}.

\subsection{Genus $0$ relative stable morphisms with full tangency}

Let us first prove some basic facts 
about genus $0$ relative stable morphisms with full tangency. 
Let $Z$ be a smooth surface 
and $D$ a smooth connected divisor on $Z$. 

\begin{definition}
(1)
For an algebraic class $\beta$ of dimension $1$ on $Z$ 
with $D.\beta>0$ 
and a non-negative integer $k$, 
let $\Gamma_{\beta, k}$ be the admissible weighted graph 
as follows. 
\begin{itemize}
\item
$\Gamma_{\beta, k}$ has only one vertex $v$ 
with $g(v)=0$ and $b(v)=\beta$. 
\item
$\Gamma_{\beta, k}$ has $k$ legs, 
and has one root of weight $D.\beta$. 
\end{itemize}
Let $\Gamma_\beta=\Gamma_{\beta, 0}$. 
We define $\fM_\beta$ to be 
$\fM(\fZrel, \Gamma_\beta)$ 
and refer to their points as genus $0$ relative stable morphisms 
of class $\beta$ with full tangency to $D$. 

(2)
We write $\fM_\beta^\circ$ for the open substack of $\fM_\beta$ 
whose points correspond to 
morphisms to $(Z[0], D[0])$. 
\end{definition}

To explain what 
a relative stable morphism $f$ in $\fM_\beta\setminus\fM_\beta^\circ$ 
looks like, 
we consider the following combinatorial data. 

\begin{definition}\label{def_graph2}
For a nonnegative integer $n$ and a positive integer $r$, 
let $\cG_{n, r}$ be the set of isomorphism classes 
of triples $(G, \lambda, \rho)$, where 
\begin{itemize}
\item
$G$ is a finite graph with the set of vertices $V(G)$, and 
\item
$\lambda: V(G)\to \{1, 2, \dots, n+1\}$  and 
$\rho: \{1, 2, \dots, r\}\to V(G)$ are maps, 
\end{itemize}
which satisfy the following: 
\begin{enumerate}
\item\label{def_graph2_ends}
If we denote $\lambda^{-1}(j)$ by $V(G)_j$, 
then $V(G)_1$ consists of one vertex, 
and $\rho$ induces a bijection 
between $\{1, 2, \dots, r\}$ and $V(G)_{n+1}$. 
\item\label{def_graph2_adjacency}
$G$ has no loops. 
Every vertex $v\in V(G)_j$ is adjacent 
only to vertices of $V(G)_{j\pm 1}$. 
Write $v\rightarrowtail w$ 
when $v$ and $w\in V(G)_{j-1}$ are adjacent. 
If $j>1$, there exists a unique vertex $w$ 
with $v\rightarrowtail w$. 
If $ j\leq n$, 
there exists at least one vertex $u$ with 
$u\rightarrowtail v$. 
\item\label{def_graph2_stability}
If $j\leq n$, 
there exists a vertex $v\in V(G)_j$ 
such that $\#\{u\in V(G)| u\rightarrowtail v\}\geq 2$. 
\end{enumerate}
\end{definition}

\begin{lemma}\label{lem_configuration}
Let $[f]\in\fM_\beta$ be a relative stable morphism 
represented by a morphism $f: X\to Z[n]_0$ and a distinguished point $Q\in X$. 
Assume that $\varphi(f(X))\not\supset D$, and 
denote the connected components 
of $\overline{X\setminus (\varphi\circ f)^{-1}(D)}$ 
by $X_1, \dots, X_r$ and 
the irreducible components 
of $\overline{X\setminus(X_1\cup\dots\cup X_r)}$ 
by $X_{r+1}, \dots, X_{r+s}$. 
Let $\beta_i=(\varphi\circ f)_*([X_i])$ for $1\leq i\leq r$. 
Then we have the following. 
\begin{enumerate}
\item 
$D.\beta_i>0$ for $1\leq i\leq r$. 
\item
The morphism $(\varphi\circ f)|_{X_i}$ belongs to $\fM_{\beta_i}^\circ$ 
for $1\leq i\leq r$. 
\item
$\varphi(f(X))\cap D$ consists of one point $P$. 
\item
There exist an element $(G, \lambda, \rho)$ of $\cG_{n, r}$ 
and a bijection 
$I: V(G)\to \{1, 2, \dots, r+s\}$ 
such that the following hold. 
For brevity, write $X_v$ for $X_{I(v)}$. 
\begin{enumerate}
\item 
$f(X_v)\subset \Delta_{\lambda(v)}$ and $X_{\rho(i)}=X_i$ for $1\leq i\leq r$. 
\item
Via the correspondence $v\mapsto X_v$, 
edges of $G$ correspond to intersection points. 
\item
Define $\mu: V(G)\to\ZZ_{>0}$ inductively as follows: 
\begin{itemize}
\item
If $v=\rho(i)\in V(G)_{n+1}$, 
then $\mu(v):=D.\beta_i$. 
\item
If $\mu$ is defined on $V(G)_{j+1}$ and $v\in V(G)_j$, 
$\mu(v):=\sum_{u\rightarrowtail v}\mu(u)$. 
\end{itemize}
Let $v$ be an element of $V(G)_j$. 
If $j>1$, 
write $R_v=X_v\cap X_w$ where $w$ is the unique vertex 
with $v\rightarrowtail w$. 
Then 
\[
\hbox{$f|_{X_v}^*\bB_j=\sum_{u\rightarrowtail v}\mu(u)R_u$ 
if $j\leq n$}, 
\]
\[
\hbox{$f|_{X_v}^*\bB_{j-1} =\mu(v)R_v$ if $j>1$}, 
\]
and 
\[
\hbox{$f|_{X_v}^*D[n]=\mu(v)Q$ if $j=1$}. 
\] 
\end{enumerate}
\end{enumerate}
\end{lemma}

\begin{proof}
We may assume that $f(X_i)$ is a point if and only if $i>r+s'$. 
Write $X'=X_1\cup\dots \cup X_{r+s'}$. 
Let $G$ be a graph whose vertices are in one-to-one 
correspondence with $\{X_1, \dots, X_{r+s'}\}$ 
and whose edges correspond to intersection points. 
By the predeformability, 
each $f(X_v)$ is contained in exactly one of $\Delta_j$. 
Define $\lambda$ and $\rho$ so that they satisfy 
$f(X_v)\in\Delta_{\lambda(v)}$ and $X_{\rho(i)}=X_i$ for $1\leq i\leq r$. 

The assertions (4a) and (4b) are clear from the definition. 
We have to prove (1), (2), (3), (4c) 
and 
\begin{center}
(4') $X=X'$, and 

(4'') $G\in\cG_{n, r}$. 
\end{center}

\begin{claim}\label{claim_configuration}
(1)
If $\lambda(v)\leq n$, 
the image $f(X_v)$ is a fiber of $\Delta_{\lambda(v)}\to D$. 
Consequently, $f(X_v)$ intersects both ``the upper boundary'' $\bB_{\lambda(v)-1}$ or $D[n]$ and 
``the lower boundary'' $\bB_{\lambda(v)}$. 

(2)
Each of $f(X_1), \dots, f(X_r)$ intersects $\bB_n$. 
\end{claim}
\begin{proof}
(1)
This follows from the assumption that 
$X$ does not dominate $D$. 

(2)
Since $D.\beta>0$, 
$\varphi(f(X))$ intersects $D$. 
If $\varphi(f(X_i))\cap D=\emptyset$ for some $i\leq r$, then $X\not= X_i$, 
and there would be an irreducible component $Y$ 
which intersects $X_i$ in a point. 
Since $Y\cap X_i$ is not contained in $(\varphi\circ f)^{-1}(D)$, 
$Y\cup X_i$ is a connected subset of 
$\overline{X\setminus (\varphi\circ f)^{-1}(D)}$, 
which is a contradiction. 
\end{proof}
In particular, the assertion (1) holds. 

From the condition that $f^{-1}(D[n])$ consists of one smooth point $Q$, 
it follows that there is a unique vertex, say $v_1$, 
such that $f(X_{v_1})\subset\Delta_1$. 
Combining this with the definitions of $\lambda$ and $\rho$, 
we see that the condition (\ref{def_graph2_ends}) 
of Definition \ref{def_graph2} is satisfied. 

From the configuration of $Z[n]_0$, 
it is clear that $v\in V(G)_j$ can be  
adjacent only to vertices in $V(G)_{j-1}\cup V(G)_j\cup V(G)_{j+1}$. 
If $Q\in X_v$ maps to $\bB_{j-1}$(resp. $\bB_j$), 
then it is also on a component of $X_w$ 
with $w\in V(G)_{j-1}$(resp. $w\in V(G)_{j+1}$)  
by the predeformability. 
Going up, 
any component is eventually connected to $X_{v_1}$. 
Thus $G$, or equivalently $X'$, is connected. 
Furthermore, 
since $G$ does not contain loops, 
a vertex $v\in V(G)_j$ is adjacent only to 
vertices in $V(G)_{j\pm 1}$, 
and 
$v$ is adjacent to a unique vertex in $V(G)_{j-1}$ if $j>1$. 
By Claim \ref{claim_configuration}, 
$f(X_v)$ intersects $\bB_{j-1}$(resp. $\bB_j$) if $j>1$(resp. if $j\leq n$), 
so 
the condition (\ref{def_graph2_adjacency}) 
of Definition \ref{def_graph2} is satisfied.

From the connectedness of $G$ and Claim \ref{claim_configuration}, 
we see that $\varphi(f(X_v))\cap D$ 
consists of the same point $\varphi(f(Q))$. 
This is also true for $X_i$ with $i>r+s'$, hence the assertion (3). 
If $\lambda(v)=n+1$, 
then $(\varphi\circ f)|_{X_v}^{-1}(D)$ 
consists of one point $R_v$. 
It is easy to see that $\varphi\circ f|_{X_v}$ is a stable morphism to $Z$, 
so the assertion (2) holds. 

We see the assertion (4c) by a descending induction on $j$. 
If $j=n+1$, the assertion is clear. 
Assume that it is true for $j+1$. 
Then the first equality for $j$ follows from the predeformability. 
We already know that $f|_{X_v}^*\bB_{j-1}$ (resp. $f|_{X_v}^*D[n]$) 
has one point $R_v$ (resp. $Q$) as the support. 
The multiplicity is the degree of $X_v\to f(X_v)$, 
so it is equal to $\sum_{u\rightarrowtail v}\mu(u)$. 

Assume that $X\not=X'$ 
and let $X''$ be a connected component of 
$\overline{X\setminus X'}$. 
Since $p_a(X)=0$ and $X'$ is connected, 
$X'$ and $X''$ intersect at one point. 
Thus $\deg \omega_X|_{X''}<0$, which contradicts the assumption 
that $f$ is stable as a morphism to $Z[n]$. 
Thus (4') holds. 

By the stability as a relative prestable morphism to $\fZrel$, 
it follows that 
for any $j\leq n$ 
there exists a vertex $v\in V(G)_j$ which 
is adjacent to at least two vertices in $V(G)_{j+1}$. 
This is the condition (\ref{def_graph2_stability}) 
of Definition \ref{def_graph2}, 
so (4'') holds. 
\end{proof}

\begin{corollary}\label{cor_configuration} 
Let the notations and assumptions be as in Lemma \ref{lem_configuration}. 
Let $\varphi(f(X))=C_1\cup\dots\cup C_s$ 
be the irreducible decomposition, 
and assume further that $D.C_i>0$ for all $i$. 

(1)
We have $r\geq s$, 
$X_1, \dots, X_r$ are isomorphic to $\bP^1$
and $(\varphi\circ f)|_{X_i}: X_i\to Z$ 
belongs to $\fM_{(\varphi\circ f)_*[X_i]}^\circ$. 
If $(\varphi\circ f)_* X$ is reduced, 
then $r=s$ and $(\varphi\circ f)|_{X_i}$ is 
a normalization map of the image curve. 

(2)
If $s=1$ and $(\varphi\circ f)_* X=C_1$, 
then $n=0$, $X\cong\bP^1$ and $f$ is 
a normalization map of $C_1$. 

(3)
If $s=2$ and $(\varphi\circ f)_* X=C_1+C_2$, 
then $f$ is described as follows. 
Let $d_i=D.C_i$. 
\begin{itemize}
\item
$n=1$. 
\item
$X$ is a chain of smooth rational curves $X_1$, $X_0$ and $X_2$ 
in this order. 
\item
For $i=1$ or $2$, 
$C_i$ is the image of $X_i$, 
and 
$\varphi\circ f|_{X_i}: X_i\to Z$ is 
a normalization map and belongs to $\fM_{[C_i]}^\circ$. 
We have $(\varphi\circ f|_{X_i})^*D=d_iR_i$, 
where $R_i=X_i\cap X_0$. 
Furthermore, 
$C_1\cap D$ and $C_2\cap D$ are the same point $P$. 
\item
The image of $\varphi\circ f|_{X_0}: X_0\to Z$ 
is the fiber of $\Delta_1\to D$ over $P$, 
and $(\varphi\circ f|_{X_0})^*D[1]=(d_1+d_2)Q$ for some point $Q\in X_0$ 
and $(\varphi\circ f|_{X_0})^*\bB_1=d_1R_1+d_2R_2$. 
\end{itemize}
In particular, $[f]$ is determined by $C_1$ and $C_2$. 

(4)
Assume that $(K_Z+D).C_i\geq 0$ for any $i$. 
Let $M$ be the image of the natural map 
from $\fM_{f_*[X]}$ to the Chow variety of $Z$. 
Then $(\varphi\circ f)_* X$ is isolated in $M$. 

(5)
If the assumptions of (2) and (4) or (3) and (4) are satisfied, 
then $[f]$ is isolated. 
\end{corollary}
\begin{proof}
(1)
By Lemma \ref{lem_configuration}, (2), 
$(\varphi\circ f)|_{X_i}: X_i\to Z$ is 
stable as a relative prestable morphism to $\fZrel$ 
for each $i$. 
So $(\varphi\circ f)|_{X_i}^{-1}D$ consists of one smooth point, 
and our assumption implies that 
there can be only one irreducible component 
of $X_i$ on which $f$ is nonconstant. 
From the stability, 
it follows that $X_i$ is irreducible, i.e. isomorphic to $\bP^1$. 
In particular, the number $s$ of irreducible components of $\varphi(f(X))$ 
is at most $r$. 

If $(\varphi\circ f)_* X$ is reduced, 
then $r=s$ since different $X_i$ map to different $C_j$, 
and $(\varphi\circ f)|_{X_i}$ is generically injective.

(2)
We have $r=1$ from (1). 
It is easy to see that $\cG_{n, 1}$ is empty for $n>0$, 
and that $\cG_{0, 1}$ has a unique element $(G, \lambda, \rho)$ 
with $V(G)=\{v\}$, $\lambda(v)=1$ and $\rho(1)=v$. 
The assertion follows from Lemma \ref{lem_configuration} (4). 

(3)
We have $r=2$ from (1). 
It is easy to see that $\cG_{n, 2}$ is empty for $n\not=1$, 
and $\cG_{1, 2}$ has a unique element $(G, \lambda, \rho)$ 
with $V(G)=\{v, w_1, w_2\}$, $\lambda(v)=1$, $\lambda(w_i)=2$ and 
$\rho(i)=w_i$. 
The assertion follows from Lemma \ref{lem_configuration} (4). 

(4)
Assume that $(\varphi\circ f)_*X$ deforms in $M$. 
Our assumptions hold also for small deformations, 
and we may assume that 
there is an irreducible component $C_i$ that really moves. 
By (1), we have a smooth curve $B$, 
a ruled surface $\pi: S\to B$ 
and a dominant morphism $\tilde{f}: S\to Z$ 
such that $D':=(\tilde{f}^{-1}D)_{\rm red}$ is a section over $B$ 
and the image of a ruling is in the class of $C_i$. 
We have $K_S+D'\geq \tilde{f}^*(K_Z+D)$, 
and if $F$ is a general fiber of $\pi$, 
we have 
\[
-1=(K_S+D').F\geq (K_Z+D).\tilde{f}_*F\geq 0, 
\]
a contradiction. 

(5) 
The image cycle does not deform by (4), 
and the relative stable morphism is determined by its image 
by (2) and (3). 
\end{proof}

\subsection{Preparation}

Now let the notations and assumptions 
be as in Theorem \ref{theorem_main}. 
We write $\beta$ for $\beta_1+\beta_2$ 
and assume that $d_1\leq d_2$. 

\begin{lemma}
There is a unique point $[f]\in\fM_\beta$ 
with ${\rm Im}(\varphi\circ f)=C_1\cup C_2$. 
It is isolated and has no nontrivial automorphisms, 
hence $\fM_\beta$ is a scheme of finite length at $[f]$. 
\end{lemma}

\begin{proof}
In fact, one may replace 
the assumption $(C_1.C_2)_P=\min\{D.C_1, D.C_2\}$ 
by $C_1\not=C_2$ in this lemma.

If $f$ satisfies the condition, 
then $(\varphi\circ f)_*X=C_1+C_2$ 
and Corollary \ref{cor_configuration} (3) and (5) imply that 
$[f]$ is unique and isolated. 
Consider an automorphism of $f$ 
formed by $r_1: X\to X$ and $r_2: Z[1]_0\to Z[1]_0$. 
Since $C_1\not=C_2$, 
the automorphism $r_1$ preserves the irreducible components. 
For $i=1$ and $2$, $\varphi\circ f|_{X_i}$ is generically injective 
and so $r_1|_{X_i}$ is the identity map. 
The automorphism $r_1|_{X_0}$ preserves 
$R_1$, $R_2$ and $Q$, 
so it is the identity map. 
It follows that $r_2$ is also the identity map. 

Conversely, it is easy to see that one can construct $f$ 
as in Corollary \ref{cor_configuration}, 
and that it belongs to $\fM_\beta$. 

\end{proof}

Write $\fM$ for the one-point scheme $[f]\in\fM_\beta$. 
We will study the structure of $\fM$ 
by an explicit calculation. 
As a preparation, 
let us fix coordinate systems. 

Let $(w_0, w_1)$ be \'etale coordinates 
on an affine open neighborhood $W$ of $P$ 
such that $w_1=0$ defines $D$ in $W$. 

Let $w_2$ be an inhomogeneous coordinate on $\PP^1$, 
$w_3=(w_2)^{-1}$ the coordinate near the point at infinity. 
As before, let $t$ denote the coordinate on $\bA^1$. 
Define \[
W[1]=(w_1w_2=t)\cup (w_1=tw_3) \subset W\times\PP^1\times\bA^1. 
\]
Then $W[1]$ can be patched with $(Z\setminus D)\times\bA^1$, 
and the union can be considered as an open neighborhood 
of $f(X)$ in $Z[1]$. 

In order to deal with the deformations of $X$, 
we make it stable by 
adding ordinary marked points. 
The moduli scheme $\cM_{0, 5}$ 
of $5$-marked genus $0$ stable curves 
has a point corresponding to our curve $X=X_1\cup X_0\cup X_2$, 
with markings $P_1, P_2\in X_1$, $Q\in X_0$ and $P_3, P_4\in X_2$. 
We have the following description 
for the formal neighborhood $\cM$ of $[X]\in\cM_{0, 5}$, 
the universal curve $\cX$ over $\cM$ 
and the marked sections $\cQ$ and $\cP_i$ extending $Q$ and $P_i$. 

\begin{enumerate}
\item
$\cM$ is a formal $2$-disk with coordinates $\mu_1$ and $\mu_2$. 
\item
$\cX=\cU'_1\cup\cU_1\cup\cU_2\cup\cU'_2$, 
where
\begin{itemize}
\item 
$\cU'_i=\bA^1_\cM=\{(\z{i}{0}, \mu_1, \mu_2)\}$, 
\item
$\cU_i=\{(\z{i}{1}, \z{i}{2}, \mu_1, \mu_2))| \z{i}{1}\z{i}{2}=\mu_i\}
\subset\bA^2_\cM$, 
\item
$\{\z{i}{0}\not=0\}\subset\cU'_i$ and $\{\z{i}{1}\not=0\}\subset\cU_i$ are patched 
by $\z{i}{0}=1/\z{i}{1}$, and 
\item
$\{\z{1}{2}\not=0\}\subset\cU_1$ and $\{\z{2}{2}\not=0\}\subset\cU_2$ are patched 
by $\z{1}{2}=1/\z{2}{2}$
(and $\z{1}{1}=\mu_1\z{2}{2}$, $\z{2}{1}=\mu_2\z{1}{2}$). 
\end{itemize}
\item
The marked sections are 
$\cQ: \z{1}{2}=\z{2}{2}=1$ on $\cU_1\cap\cU_2$, 
$\cP_1: \z{1}{0}=0$ on $\cU'_1$, 
$\cP_2: \z{1}{0}=\z{1}{1}=1$ on $\cU'_1\cap\cU_1$, 
$\cP_3: \z{2}{0}=\z{2}{1}=1$ on $\cU'_2\cap\cU_2$ 
and 
$\cP_4: \z{2}{0}=0$ on $\cU'_2$. 
\end{enumerate}

Let $0$ be the closed point of $\cM$ and 
denote the fiber $(\cU_i)_0$ by $U_i$, 
$(\cU'_i)_0$ by $U'_i$, 
$\cQ(0)$ by $Q$, etc. 

Write $X_i^\circ$ for $X\setminus R_i$ 
and $\cX_i^\circ$ for the corresponding 
open subspace of $\cX$. 

Denote $f^{-1}(W)$ by $V$. 
We may assume that $P_1, P_4\not\in V$. 
Write 
\begin{eqnarray*}
V_0 & = & U_1\cap U_2, \\
V_1 & = & V\cap(U_1\setminus Q), \\
V_2 & = & V\cap(U_2\setminus Q), 
\end{eqnarray*}
and denote the corresponding open subspaces of $\cX$ 
by $\cV$ and $\cV_i$. 

Let $A$ be the coordinate ring of 
the affine variety 
\[
\{(p, w_2, t)\in W\times\bA^1\times\bA^1| w_1(p)w_2=t\}\subset W[1]. 
\]
For $i=1$ or $2$, 
$f|_{V_i}$ corresponds to 
a ring homomorphism $A\to \cO_X(V_i)$ satisfying 
\begin{eqnarray*}
w_0 & \mapsto & \f{A}{i}(\z{i}{1}), 
\qquad \f{A}{i}(\z{i}{1})\in\cO_{X_i}(V\cap X_i), \\
w_1 & \mapsto & \z{i}{1}^{d_i}\f{B}{i}(\z{i}{1}), 
\qquad \f{B}{i}(\z{i}{1})\in\cO_{X_i}(V\cap X_i), \\
w_2 & \mapsto & \sigma_i\z{i}{2}^{d_i}\f{C}{i}(\z{i}{2}), 
\qquad \f{C}{i}(\z{i}{2})=\frac{1}{(1-\z{i}{2})^{d_1+d_2}}, 
\end{eqnarray*}
where $\sigma_1=1$ and $\sigma_2=(-1)^{d_1+d_2}$. 
By the assumptions, 
we have $\f{A}{1}(0)=\f{A}{2}(0)$ and 
$\f{B}{i}(0)$ is nowhere zero. 
We have $\f{A}{i}'(0)\not=0$ if $d_i>1$, 
and we assume the same also when $d_i=1$ 
by making a coordinate change. 

Let $A'$ be the coordinate ring of the affine variety 
\[
\{(p, w_3, t)\in W\times\bA^1\times\bA^1| w_1(p)=tw_3\}\subset W[1]. 
\]
Then $f|_{V_0}$ is given by 
a ring homomorphism $A'\to \cO_X(V_0)$ with 
\[
w_0\mapsto 0, w_1\mapsto 0, 
w_3\mapsto \frac{(1-\z{1}{2})^{d_1+d_2}}{\z{1}{2}^{d_1}}. 
\]

\begin{lemma}
For $1\leq j\leq 4$, 
let $H_j\subset Z[1]$ be a hypersurface 
which intersects $f(X)$ transversally at $f(P_j)$. 
Take a general section $\cQ':\cM\to\cU_1\cap\cU_2$, 
denote $\cQ'(0)$ by $Q'$ 
and let $H'$ be the hypersurface $w_2 = w_2(f(Q'))$. 

Let $\tilde{\fM}$ be the completion of 
$\fM((Z[1], D[1]), \Gamma_{\beta, 4})^{st}$ at $[(f, P_1, \dots, P_4)]$, 
$\tilde{\fM}'$ the subspace of $\tilde{\fM}$ 
formed by morphisms which map $\cP_i$ to $H_i$, 
$\tilde{\fM}''$ the subspace of $\tilde{\fM}'$ 
formed by morphisms which map $\cQ'$ to $H'$. 

Then $\tilde{\fM}\to\fM$ and $\tilde{\fM}'\to\fM$ are smooth 
of relative dimensions $5$ and $1$, 
and $\tilde{\fM}''\to\fM$ is an isomorphism. 
\end{lemma}
\begin{proof}
By Proposition \ref{prop_quotient}, 
$\fM((Z[1], D[1]), \Gamma_\beta)^{st}/\bG_m\to\fM_\beta$ is \'etale at $[f]$. 
Since $f$ has no nontrivial automorphisms, 
it is an isomorphism at $[f]$. 

At $[f, P_1, \dots, P_4]$, the morphism 
$\fM((Z[1], D[1]), \Gamma_{\beta, 4})^{st}\to\fM((Z[1], D[1]), \Gamma_\beta)^{st}$ 
is smooth of relative dimension $4$, 
corresponding to the degree of freedom of marked points. 
Thus $\tilde{\fM}\to\fM$ is smooth of 
relative dimension $5$, 
and $\tilde{\fM}'\to\fM$ is smooth 
of relative dimension $1$. 

Since $\bG_m$ acts faithfully on $f(Q')$, 
$\tilde{\fM}''\to\fM$ is an isomorphism. 
\end{proof}

So, we study how $(f, P_1, \dots, P_4)$ deforms in $\tilde{\fM}$. 
Note that infinitesimal deformations of $f|_V$ 
can be given by lifting $w_i$, 
for $w_0, w_1$ are \'etale coordinates on $W$.

\subsection{Calculation}

Consider $S=S_n:=\Spec \CC[s]/(s^{n+1})$. 
This will be sufficient 
since the tangent space to $\fM$ 
is at most $1$-dimensional 
as we will see later. 

A deformation of $[f, P_1, \dots, P_4]\in\tilde{\fM}$ over $S$ 
is given by morphisms $\mu: S\to\cM$, $\tau: S\to\bA^1$ 
and $\tilde{f}: \cX_S\to Z[1]_S$. 
Let $\mu$ and $\tau$ 
be given by 
$\mu_i=\sum_{k=1}^n \m{i}{k}s^k$ 
and $t=\sum_{k=1}^n t^{(k)}s^k$, 
where we denote the residue class of $s$ by the same symbol. 
For brevity, we write $\cX$ for $\cX_S$, etc.

Let $\w{i}{j}=\tilde{f}^*w_j|_{\cV_i}$. 
From the predeformability, 
one can write 
\[
\w{i}{1}=\z{i}{1}^{d_i}\fn{B}{i}, 
\w{i}{2}=\sigma_i \z{i}{2}^{d_i}\fn{C}{i}
\]
where $\fn{B}{i}$ and $\fn{C}{i}$ are elements of $\cO_{\cX}(\cV_i)$ 
satisfying $\fn{B}{i}\fn{C}{i}\in\CC[s]/(s^{n+1})$.  
Also write $\fn{A}{i}=\w{i}{0}$. 

We expand $\fn{A}{i}$, $\fn{B}{i}$ and $\fn{C}{i}$ as 
 \begin{eqnarray*}
\fn{A}{i} & = & \sum_{k=0}^n s^k\left\{
\fnl{A}{i}{0}{k}+\z{i}{1}\fnl{A}{i}{1}{k}(\z{i}{1})+\z{i}{2}\fnl{A}{i}{2}{k}(\z{i}{2})
\right\} \\
\fn{B}{i} & = & \sum_{k=0}^n s^k\left\{
\fnl{B}{i}{0}{k}+\z{i}{1}\fnl{B}{i}{1}{k}(\z{i}{1})+\z{i}{2}\fnl{B}{i}{2}{k}(\z{i}{2})
\right\} \\
\fn{C}{i} & = & \sum_{k=0}^n s^k\left\{
\fnl{C}{i}{0}{k}+\z{i}{1}\fnl{C}{i}{1}{k}(\z{i}{1})+\z{i}{2}\fnl{C}{i}{2}{k}(\z{i}{2})
\right\} 
\end{eqnarray*}
with 
\begin{eqnarray*}
\fnl{A}{i}{0}{k}, \fnl{B}{i}{0}{k}, \fnl{C}{i}{0}{k} & \in & \CC, \\
\fnl{A}{i}{1}{k}(\z{i}{1}), \fnl{B}{i}{1}{k}(\z{i}{1}), \fnl{C}{i}{1}{k}(\z{i}{1})
 & \in & \cO_{X_i}(V_i\cap X_i)=\cO_{X_i}(V\cap X_i), \\
\fnl{A}{i}{2}{k}(\z{i}{2}), \fnl{B}{i}{2}{k}(\z{i}{2}), \fnl{C}{i}{2}{k}(\z{i}{2})
 & \in & \cO_{X_0}(V_i\cap X_0)=\cO_{X_0}((U_i\cap X_0)\setminus Q). 
\end{eqnarray*}

Let us write down the conditions. 

\begin{lemma}\label{lem_conditions}
The above data give a family in $\tilde{\fM}$ 
if and only if the following hold. 
\begin{enumerate}
\item\label{central_fiber}
(Central fiber)\qquad
Restricting to $s=0$, we have $f$. 
\item\label{w0_on_V0}
($w_0$ on $\cV_0$)\qquad
$\w{1}{0}|_{\cV_0\setminus Q}=\w{2}{0}|_{\cV_0\setminus Q}$, 
and it extends to $Q$. 
\item\label{w2_on_V0}
($w_2$ on $\cV_0$)\qquad
$\w{1}{2}|_{\cV_0\setminus Q}=\w{2}{2}|_{\cV_0\setminus Q}$, 
and its inverse extends to a function on $\cV_0$ 
with vanishing order $d_1+d_2$ at $Q$. 
\item\label{w0w1_on_X1X2}
($w_0$ and $w_1$ on $\cX_1^\circ$, $\cX_2^\circ$)\qquad
The morphism $\cV\cap \cX_i^\circ\to W$ given by 
$\w{i}{0}$ and $\w{i}{1}$ 
extends to a morphism $\cX_i^\circ\to Z$. 
\item\label{w1_on_V0_w2_on_X1X2}
($w_1$ on $\cV_0$, $w_2$ on $\cX_1^\circ$, $\cX_2^\circ$)\qquad
$\w{1}{1}|_{\cV_0\setminus Q}=\w{2}{1}|_{\cV_0\setminus Q}$, 
and it extends to $Q$. 
The function $\w{i}{2}|_{\cV\cap \cX_i^\circ}$ 
extends to a function on $\cX_i^\circ$. 
\item\label{mapped_into_Z1}
(Mapped into $Z[1]$)\qquad
$\w{1}{1}\w{1}{2}=\w{2}{1}\w{2}{2}=t$. 
\item\label{predeformability}
(Predeformability)\qquad
$\fn{B}{i}\fn{C}{i}\in\CC[s]/(s^{n+1})$. 
\end{enumerate}
\end{lemma}

\subsection{Central fiber}

Setting $s=0$, the condition (\ref{central_fiber}) 
of Lemma \ref{lem_conditions} is equivalent to 
\begin{eqnarray*}
\fnl{A}{i}{0}{0}+\z{i}{1}\fnl{A}{i}{1}{0}(\z{i}{1})+\z{i}{2}\fnl{A}{i}{2}{0}(\z{i}{2}) 
& = & \f{A}{i}(\z{i}{1}), \\
\fnl{B}{i}{0}{0}+\z{i}{1}\fnl{B}{i}{1}{0}(\z{i}{1}) & = & \f{B}{i}(\z{i}{1}), \\
\fnl{C}{i}{0}{0}+\z{i}{2}\fnl{C}{i}{2}{0}(\z{i}{2}) & = & \f{C}{i}(\z{i}{2}). 
\end{eqnarray*}

\subsection{Predeformability modulo $s$}

Modulo $s$, we have
\begin{eqnarray*}
\fn{B}{i}\fn{C}{i} & \equiv &
\f{B}{i}(\z{i}{1})\f{C}{i}(\z{i}{2}) 
+ \z{i}{2}\fnl{B}{i}{2}{0}(\z{i}{2})\f{C}{i}(\z{i}{2})
+\f{B}{i}(\z{i}{1})\z{i}{1}\fnl{C}{i}{1}{0}(\z{i}{1}) \\
& \equiv & 
\f{B}{i}(\z{i}{1})\f{C}{i}(0) 
+\f{B}{i}(0)\f{C}{i}(\z{i}{2}) 
-\f{B}{i}(0)\f{C}{i}(0) \\
& & 
+ \f{C}{i}(\z{i}{2})\z{i}{2}\fnl{B}{i}{2}{0}(\z{i}{2})
+\f{B}{i}(\z{i}{1})\z{i}{1}\fnl{C}{i}{1}{0}(\z{i}{1}). 
\end{eqnarray*}
This must be a constant by the predeformability 
(Lemma \ref{lem_conditions}, (\ref{predeformability})), 
so we have the following. 

\begin{lemma}\label{lem_predeformability_0}
The predeformability holds modulo $s$ if and only if 
\begin{eqnarray*}
\fnl{B}{i}{2}{0}(\z{i}{2}) & = & 
-\f{B}{i}(0)\frac{\f{C}{i}(\z{i}{2})-\f{C}{i}(0)}{\z{i}{2}\f{C}{i}(\z{i}{2})}, \\
\fnl{C}{i}{1}{0}(\z{i}{1}) & = & 
-\f{C}{i}(0)\frac{\f{B}{i}(\z{i}{1})-\f{B}{i}(0)}{\z{i}{1}\f{B}{i}(\z{i}{1})}. 
\end{eqnarray*}
In particular, $\fnl{B}{i}{2}{0}(0)=-\f{B}{i}(0)\f{C}{i}'(0)/\f{C}{i}(0)\not=0$. 
\end{lemma}

Thus the $0$-th order terms 
$\fnl{A}{i}{j}{0}$, $\fnl{B}{i}{j}{0}$ and $\fnl{C}{i}{j}{0}$ 
are uniquely determined.

\subsection{Extending deformations}

In this and the following subsections, 
we assume that $\{\fnl{A}{i}{j}{k}, \fnl{B}{i}{j}{k}, \fnl{C}{i}{j}{k}\}_{k<n}$ 
satisfies the conditions modulo $s^n$. 
Denote by $\tilde{f}'$ the corresponding family of relative stable morphisms 
over $S_{n-1}$. 
Let us rewrite the conditions of Lemma \ref{lem_conditions} 
for $\fnl{A}{i}{j}{n}$, $\fnl{B}{i}{j}{n}$ and $\fnl{C}{i}{j}{n}$.

\begin{lemma}\label{lem_w0_on_V0}
For any $\m{1}{n}$ and $\m{2}{n}$, 
the condition (\ref{w0_on_V0}) of Lemma \ref{lem_conditions} is satisfiable 
if and only if $\fnl{A}{1}{0}{n}=\fnl{A}{2}{0}{n}$, 
and in that case 
there exist uniquely 
$\fnl{A}{1}{2}{n}(\z{1}{2})$ and $\fnl{A}{2}{2}{n}(\z{2}{2})$
such that the condition is satisfied 
for any choice of the remaining data. 
\end{lemma}
\begin{proof}
On $V_0$, we have $\z{i}{1}=\sum_{k=1}^n \m{i}{k}s^k/\z{i}{2}$, 
and one has 
\begin{eqnarray*}
\fn{A}{i}|_{\cV_0\setminus Q}  & = & 
s
\left\{
\frac{1}{\z{i}{2}}R^{(1)}_{A, i}(\z{i}{2}) 
+ \fnl{A}{i}{0}{1} +\z{i}{2}\fnl{A}{i}{2}{1}(\z{i}{2})
\right\} \\
& & + \dots \\
& & + s^n\left\{
\frac{1}{\z{i}{2}}R^{(n)}_{A, i}(\z{i}{2}) 
+ \fnl{A}{i}{0}{n} + \z{i}{2}\fnl{A}{i}{2}{n}(\z{i}{2})
\right\}, 
\end{eqnarray*}
where 
\[
\sum_{k=1}^n s^k R^{(k)}_{A, i}(\z{i}{2}) = 
\sum_{k=1}^n \m{i}{k}s^k
\sum_{k=0}^{n-1} 
s^k \fnl{A}{i}{1}{k}\left(\frac{\sum_{k'=1}^n \m{i}{k'}s^{k'}}{\z{i}{2}}\right). 
\]
From $s^{n+1}=0$, 
it follows that $R^{(n)}_{A, i}(\z{i}{2})$ belongs to $\CC[1/\z{i}{2}]$ 
and is determined by 
$\{\m{i}{k}\}_{k\leq n}$ and $\{\fnl{A}{i}{1}{k}\}_{k<n}$. 
By the requirement that $\w{i}{0}|_{\cV_0\setminus Q}$ extends to $\z{i}{2}=1$, 
we have $\fnl{A}{i}{2}{n}(\z{i}{2})\in\CC[\z{i}{2}]$. 
The assertion follows 
from the requirement 
$\w{1}{0}|_{\cV_0\setminus Q}=\w{2}{0}|_{\cV_0\setminus Q}$. 
\end{proof}

\begin{lemma}\label{lem_w2_on_V0}
For any $\m{1}{n}$ and $\m{2}{n}$, 
there exist 
$\fnl{C}{1}{0}{n}$, $\fnl{C}{2}{0}{n}$, 
$\fnl{C}{1}{2}{n}(\z{1}{2})$ and $\fnl{C}{2}{2}{n}(\z{2}{2})$
(with a freedom of dimension $1$) 
such that the condition (\ref{w2_on_V0}) of Lemma \ref{lem_conditions} is satisfied 
for any choice of the remaining data. 
\end{lemma}
\begin{proof}
Let us introduce $\fnl{D}{i}{j}{k}$ by 
\[
(1-\z{i}{2})^{d_1+d_2}\fn{C}{i}
= 
\sum_{k=0}^n s^k\left\{
\fnl{D}{i}{0}{k}+\z{i}{1}\fnl{D}{i}{1}{k}(\z{i}{1})+\z{i}{2}\fnl{D}{i}{2}{k}(\z{i}{2})
\right\}. 
\]
Note that $\{\fnl{C}{i}{j}{k}\}_{j=0, 1, 2, k\leq m}$ and 
$\{\fnl{D}{i}{j}{k}\}_{j=0, 1, 2, k\leq m}$ determine each other. 

We have 
\begin{eqnarray*}
& & \w{1}{2}|_{\cV_0\setminus Q} \\
& = & 
\frac{(\z{1}{2})^{d_1}}{(1-\z{1}{2})^{d_1+d_2}}
\sum_{k=0}^n s^k\left\{
\fnl{D}{1}{0}{k}+\z{1}{1}\fnl{D}{1}{1}{k}(\z{1}{1})+\z{1}{2}\fnl{D}{1}{2}{k}(\z{1}{2})
\right\} \\
& = & 
\frac{1}{(1-\z{1}{2})^{d_1+d_2}}
\sum_{k=0}^n s^k\left\{
\left(\sum_{k'=1}^n \m{1}{k'}s^{k'}\right){\z{1}{2}}^{d_1-1}
\fnl{D}{1}{1}{k}\left(\left(\sum_{k'=1}^n \m{1}{k'}s^{k'}\right)\frac{1}{\z{1}{2}}\right) 
\right. \\
& & 
\left. 
\qquad + {\z{1}{2}}^{d_1}\fnl{D}{1}{0}{k}
+{\z{1}{2}}^{d_1+1}
\fnl{D}{1}{2}{k}(\z{1}{2})
\right\} \\
& = & 
\frac{1}{(1-\z{1}{2})^{d_1+d_2}}
\sum_{k=0}^n s^k\left\{
{\z{1}{2}}^{d_1-1}R^{(k)}_{D, 1}(\z{1}{2}) 
+ {\z{1}{2}}^{d_1}\fnl{D}{1}{0}{k} 
+ {\z{1}{2}}^{d_1+1}\fnl{D}{1}{2}{k}(\z{1}{2}) 
\right\}  
\end{eqnarray*}
and
\begin{eqnarray*}
& & \w{2}{2}|_{\cV_0\setminus Q} \\
& = & 
\sigma_2\frac{(\z{2}{2})^{d_2}}{(1-\z{2}{2})^{d_1+d_2}}
\sum_{k=0}^n s^k\left\{
\fnl{D}{2}{0}{k}+\z{2}{1}\fnl{D}{2}{1}{k}(\z{2}{1})+\z{2}{2}\fnl{D}{2}{2}{k}(\z{2}{2})
\right\} \\
& = & 
\frac{(\z{1}{2})^{d_1}}{(1-\z{1}{2})^{d_1+d_2}}
\sum_{k=0}^n s^k\left\{
\fnl{D}{2}{0}{k}+\z{2}{1}\fnl{D}{2}{1}{k}(\z{2}{1})+\z{2}{2}\fnl{D}{2}{2}{k}(\z{2}{2})
\right\} \\
& = & 
\frac{1}{(1-\z{1}{2})^{d_1+d_2}}
\sum_{k=0}^n s^k\left\{
{\z{1}{2}}^{d_1-1}\fnl{D}{2}{2}{k}\left(\frac{1}{\z{1}{2}}\right) + {\z{1}{2}}^{d_1}\fnl{D}{2}{0}{k}\right. \\
& & 
\left. 
\qquad +\left(\sum_{k'=1}^n \m{2}{k'}s^{k'}\right){\z{1}{2}}^{d_1+1}
\fnl{D}{2}{1}{k}\left(\left(\sum_{k'=1}^n \m{2}{k'}s^{k'}\right)\z{1}{2}\right)
\right\} \\
& = & 
\frac{1}{(1-\z{1}{2})^{d_1+d_2}}
\sum_{k=0}^n s^k\left\{
{\z{1}{2}}^{d_1-1}\fnl{D}{2}{2}{k}\left(\frac{1}{\z{1}{2}}\right)
+ {\z{1}{2}}^{d_1}\fnl{D}{2}{0}{k}
+ {\z{1}{2}}^{d_1+1}R^{(k)}_{D, 2}(\z{1}{2}) 
\right\}, 
\end{eqnarray*}
where 
\[
\sum_{k=0}^n s^k R^{(k)}_{D, 1}(\z{1}{2}) = 
\sum_{k=1}^n \m{1}{k}s^k
\sum_{k=0}^{n-1} 
s^k \fnl{D}{1}{1}{k}\left(\frac{\sum_{k'=1}^n \m{1}{k'}s^{k'}}{\z{1}{2}}\right)
\]
and 
\[
\sum_{k=0}^n s^k R^{(k)}_{D, 2}(\z{1}{2}) = 
\sum_{k=1}^n \m{2}{k}s^k
\sum_{k=0}^{n-1} 
s^k \fnl{D}{2}{1}{k}\left(\z{1}{2}\sum_{k'=1}^n \m{2}{k'}s^{k'}\right). 
\]
We see that $R^{(n)}_{D, i}(\z{1}{2})$ is determined by 
$\{\m{i}{k}\}_{k\leq n}$ and $\{\fnl{D}{i}{1}{k}\}_{k<n}$, 
and that $R^{(n)}_{D, 1}(\z{1}{2})\in \CC[1/\z{1}{2}]$ 
and $R^{(n)}_{D, 2}(\z{1}{2})\in \CC[\z{1}{2}]$. 
The condition that $(\w{i}{2}|_{\cV_0\setminus Q})^{-1}$ 
has vanishing order $d_1+d_2$ at $Q$ 
is equivalent to $\fnl{D}{i}{2}{n}(\z{i}{2})\in\CC[\z{i}{2}]$. 
From $\w{1}{2}|_{\cV_0\setminus Q}=\w{2}{2}|_{\cV_0\setminus Q}$, 
it follows that 
$\fnl{D}{1}{0}{n}=\fnl{D}{2}{0}{n}$ 
and that 
$\fnl{D}{1}{2}{n}(\z{1}{2})$ and $\fnl{D}{2}{2}{n}(\z{2}{2})$
are uniquely determined. 
To go back to $\fnl{C}{i}{j}{n}$, 
note that 
\[
\fnl{C}{i}{0}{n}+\z{i}{2}\fnl{C}{i}{2}{n}(\z{i}{2})-
\frac{1}{(1-\z{i}{2})^{d_1+d_2}}
\left\{\fnl{D}{i}{0}{n}+\z{i}{2}\fnl{D}{i}{2}{n}(\z{i}{2})
\right\}
\]
and 
\[
\fnl{C}{i}{1}{n}(\z{i}{1})- 
\fnl{D}{i}{1}{n}(\z{i}{1}) 
\]
are functions of 
$\{\m{i}{k}\}_{k\leq n}$ and $\{\fnl{D}{i}{1}{k}\}_{k<n}$. 
Thus we obtain the result. 
\end{proof}

\begin{lemma}\label{lem_w0w1_on_X1X2}
(1)
For $n=1$, 
the condition (\ref{w0w1_on_X1X2}) of Lemma \ref{lem_conditions} 
is satisfied if and only if 
there exists a polynomial $c_i(\z{i}{1})$ of degree at most $2$ 
such that 
\begin{eqnarray*}
\fnl{A}{i}{0}{1}+\z{i}{1}\fnl{A}{i}{1}{1}(\z{i}{1})
& = & 
c_i(\z{i}{1})(\f{A}{i}(\z{i}{1}))', 
\\
\z{i}{1}^{d_i-1}\left\{
\m{i}{1}\fnl{B}{i}{2}{0}(0)
+\z{i}{1}\fnl{B}{i}{0}{1}+\z{i}{1}^2\fnl{B}{i}{1}{1}(\z{i}{1})
\right\}
& = & 
c_i(\z{i}{1})(\z{i}{1}^{d_i}\f{B}{i}(\z{i}{1}))'. 
\end{eqnarray*}

(2)
For $n>1$, let $\fnl{A}{i}{0}{n}$ be arbitrarily fixed. 
Then one can find $\m{i}{n}$, $\fnl{A}{i}{1}{n}(\z{i}{1})$, $\fnl{B}{i}{0}{n}$ 
and $\fnl{B}{i}{1}{n}(\z{i}{1})$ 
for which the condition (\ref{w0w1_on_X1X2}) of Lemma \ref{lem_conditions} 
is satisfied. 
\end{lemma}

\begin{proof}
On $\cX_i^\circ$, 
we have $\z{i}{2}=\sum_{k=1}^n \m{i}{k}s^k/\z{i}{1}$, 
\begin{eqnarray*}
\w{i}{0}|_{\cV\cap \cX_i^\circ} 
& = & 
\f{A}{i}(\z{i}{1})
+ \sum_{k=1}^n s^k\left\{
\frac{1}{\z{i}{1}}Q^{(k)}_{A, i}(\z{i}{1})+\fnl{A}{i}{0}{k}+\z{i}{1}\fnl{A}{i}{1}{k}(\z{i}{1})
\right\} 
\end{eqnarray*}
and 
\begin{eqnarray*}
& & \w{i}{1}|_{\cV\cap \cX_i^\circ}  \\
& = & \z{i}{1}^{d_i}\fn{B}{i} \\
& = & 
\z{i}{1}^{d_i}\f{B}{i}(\z{i}{1}) \\
& & 
+ \sum_{k=1}^n s^k\left[
\z{i}{1}^{d_i-1}\left\{Q^{(k)}_{B, i}(\z{i}{1})
+ \m{i}{k}\fnl{B}{i}{2}{0}(0)\right\}
+ \z{i}{1}^{d_i}\fnl{B}{i}{0}{k}+\z{i}{1}^{d_i+1}\fnl{B}{i}{1}{k}(\z{i}{1})
\right], 
\end{eqnarray*}
where $Q^{(m)}_{A, i}(\z{i}{1})$ and $Q^{(m)}_{B, i}(\z{i}{1})$ 
are elements of $\CC[1/{\z{i}{1}}]$ determined by 
$\{\m{i}{k}, \fnl{A}{i}{j}{k}\}_{k<m}$ and 
$\{\m{i}{k}, \fnl{B}{i}{j}{k}\}_{k<m}$, respectively. 
For $m=1$, we have $Q^{(1)}_{A, i}(\z{i}{1})=Q^{(1)}_{B, i}(\z{i}{1})=0$. 

Let $f_i$ be the restriction of $\varphi\circ f$ to $X_i$. 
The extensions of $\varphi\circ\tilde{f}'|_{\cX_i^\circ}$ 
to the $n$-th order correspond 
to sections of a torsor over $(f_i^*T_Z)|_{X_i^\circ}$. 
To study the extensions near $R_i$, 
we look at the following sheaf. 

\begin{lemma}
Let $\cE_i$ be the subsheaf of $f_i^*T_Z$ 
spanned by $f_i^*T_Z(-\log D)$ and $T_{X_i}$. 

(1)
The sheaf $\cE_i$ can also be described as the sheaf 
spanned by $f_i^*T_Z(-\log D)$ and $(f_i^*w_1/\z{i}{1})\partial_{w_1}$. 

(2)
There is a commutative diagram with exact rows and columns as follows. 
\[
\xymatrix{
 & 0 \ar[d] & 0 \ar[d] & & \\
0 \ar[r] & T_{X_i}(-\log R_i) \ar[r]\ar[d] & 
f_i^*T_Z(-\log D)\ar[r]\ar[d] & \cO(-1) \ar[r]\ar@{=}[d] & 0 \\
0 \ar[r] & T_{X_i} \ar[r]\ar[d] & \cE_i \ar[r]\ar[d] & \cO(-1) \ar[r] & 0 \\
 & \CC \ar@{=}[r]\ar[d] & \CC\ar[d] & & \\
 & 0 & 0 & & \\
}
\]

(3)
We have $H^1(\cE_i)=0$, 
and the natural map 
$\Gamma(T_{X_i})\to\Gamma(\cE_i)$ is an isomorphism. 
\end{lemma}
\begin{proof}
(1) 
Direct calculations. 

(2)
We have the exact sequence in the first row 
since $f_i$ is immersive, 
and the assertion is easy to prove. 

(3) follows from (2). 
\end{proof}

\begin{lemma}
Let $\cF_i$ be the sheaf on $X_i$ defined as follows: 
$\cF_i(U)$ is the set of extensions of $\tilde{f}'|_{U\cap \cX_i^\circ}$ 
to the $n$-th order 
whose expansion near $R_i$ is given by 
\begin{eqnarray*}
w_0
& = & 
\f{A}{i}(\z{i}{1})
+ \sum_{k=1}^{n-1} s^k\left\{
\frac{1}{\z{i}{1}}Q^{(k)}_{A, i}(\z{i}{1})+\fnl{A}{i}{0}{k}+\z{i}{1}\fnl{A}{i}{1}{k}(\z{i}{1})
\right\} \\
& & 
+ s^n\left\{
\frac{1}{\z{i}{1}}Q^{(n)}_{A, i}(\z{i}{1})+W_0(\z{i}{1})
\right\}
\end{eqnarray*}
and 
\begin{eqnarray*}
w_1 & = & 
\z{i}{1}^{d_i}\f{B}{i}(\z{i}{1}) \\
& & 
+ \sum_{k=1}^{n-1} s^k\left[
\z{i}{1}^{d_i-1}\left\{Q^{(k)}_{B, i}(\z{i}{1})
+ \m{i}{k}\fnl{B}{i}{2}{0}(0)\right\}
+ \z{i}{1}^{d_i}\fnl{B}{i}{0}{k}+\z{i}{1}^{d_i+1}\fnl{B}{i}{1}{k}(\z{i}{1})
\right] \\
& & 
+ s^n
\z{i}{1}^{d_i-1}\left\{Q^{(n)}_{B, i}(\z{i}{1})
+ W_1(\z{i}{1})
\right\} 
\end{eqnarray*}
for some 
$W_0(\z{i}{1}), W_1(\z{i}{1})\in\CC[[\z{i}{1}]]$. 

Then $\cF_i$ is a torsor over $\cE_i$. 
If $n=1$, $\cF_i$ is naturally isomorphic to $\cE_i$. 
\end{lemma}
\begin{proof}
In the usual $T_{X_i^\circ}$-torsor structure 
on $\cF_i|_{X_i^\circ}$, 
the action of a vector 
$a_0(\z{i}{1})\partial_{w_0} + \z{i}{1}^{d_i-1}a_1(\z{i}{1})\partial_{w_1}$ 
is given by adding $a_0(\z{i}{1})$ and $a_1(\z{i}{1})$ 
to $W_0(\z{i}{1})$ and $W_1(\z{i}{1})$. 
So $\cF_i$ is a torsor over $\cE_i$. 
For $n=1$, we have $Q^{(1)}_{B, i}(\z{i}{1})=0$ 
and so the natural isomorphism 
$\cF_i|_{X_i^\circ}\cong T_{X_i^\circ}$ gives 
$(W_0(\z{i}{1}), W_1(\z{i}{1}))\mapsto 
W_0(\z{i}{1})\partial_{w_0} + \z{i}{1}^{d_i-1}W_1(\z{i}{1})\partial_{w_1}$ 
near $R_i$. 
Thus we have a natural isomorphism 
$\cF_i\cong\cE_i$. 
\end{proof}

Let us return to the proof of Lemma \ref{lem_w0w1_on_X1X2}. 
For $n=1$, 
the vector $v_i$ defined as 
\[
\left\{
\fnl{A}{i}{0}{1}+\z{i}{1}\fnl{A}{i}{1}{1}(\z{i}{1})
\right\}\partial_{w_0}
+\z{i}{1}^{d_i-1}\left\{
\m{i}{1}\fnl{B}{i}{2}{0}(0)
+\z{i}{1}\fnl{B}{i}{0}{1}+\z{i}{1}^2\fnl{B}{i}{1}{1}(\z{i}{1})
\right\}\partial_{w_1} 
\]
must extend to a section of $\cE_i$, 
and so comes from a section of $T_{X_i}$. 
In other words, there exists a polynomial $c_i(\z{i}{1})$ of degree at most $2$ 
such that $v_i=c_i(\z{i}{1})(f_i)_*\partial_{\z{i}{1}}$. 

For $n>1$, 
$\cF_i$ is a trivial torsor since $H^1(\cE_i)=0$. 
So it has a section, 
and by adding an appropriate multiple of $(f_i)_*\partial_{\z{i}{1}}$, 
we have a section with the desired $\fnl{A}{i}{0}{n}$. 
This determines $W_0(\z{i}{1})$ and $W_1(\z{i}{1})$ in the previous lemma. 
Since $\fnl{B}{i}{2}{0}(0)\not=0$ by Lemma \ref{lem_predeformability_0},  
we can find $\m{i}{n}$, $\fnl{A}{i}{1}{n}(\z{i}{1})$, $\fnl{B}{i}{0}{n}$ 
and $\fnl{B}{i}{1}{n}(\z{i}{1})$ so that 
$\fnl{A}{i}{0}{n}+\z{i}{1}\fnl{A}{i}{1}{n}(\z{i}{1})=W_0(\z{i}{1})$ and 
$\m{i}{n}\fnl{B}{i}{2}{0}(0)+\z{i}{1}\fnl{B}{i}{0}{n}+\z{i}{1}^2\fnl{B}{i}{1}{n}(\z{i}{1})=W_1(\z{i}{1})$ hold. 
\end{proof}

\begin{lemma}\label{lem_w1_on_V0_w2_on_X1X2}
If $n<d_1$, 
the condition (\ref{w1_on_V0_w2_on_X1X2}) of Lemma \ref{lem_conditions} 
is always satisfied. 
\end{lemma}
\begin{proof}
In fact, we have $\z{i}{1}^{d_i}=0$ on $\cV_0$ 
and $\z{i}{2}^{d_i}=0$ on $\cX_i^\circ$. 
\end{proof}

\begin{lemma}\label{lem_mapped_into_Z1}
If $n<d_1$, 
the condition (\ref{mapped_into_Z1}) of Lemma \ref{lem_conditions} 
is always satisfied for $t=0$. 
If $n\geq d_1$ and $d_1=d_2$, it is satisfiable only if 
$(\m{1}{1})^{d_1}\f{B}{1}(0)\f{C}{1}(0)=(\m{2}{1})^{d_1}\f{B}{2}(0)\f{C}{2}(0)$. 
If $n\geq d_1$ and $d_1<d_2$, it is satisfiable only if $\m{1}{1}=0$. 
\end{lemma}
\begin{proof}
Immediate from 
\[
\w{i}{1}\w{i}{2}=\sigma_i(\z{i}{1}\z{i}{2})^{d_i}\fn{B}{i}\fn{C}{i}
\equiv 
 \sigma_i(\m{i}{1})^{d_i}\f{B}{i}(0)\f{C}{i}(0)s^{d_i}
 \mod s^{d_i+1}. 
\]
\end{proof}

\begin{lemma}\label{lem_predeformability}
If $\m{i}{n}$, $\fnl{B}{i}{0}{n}$, $\fnl{C}{i}{0}{n}$, 
$\fnl{B}{i}{1}{n}(\z{i}{1})$ and $\fnl{C}{i}{2}{n}(\z{i}{2})$ are fixed, 
there are unique $\fnl{B}{i}{2}{n}(\z{i}{2})$ and $\fnl{C}{i}{1}{n}(\z{i}{1})$ 
such that the condition (\ref{predeformability}) of Lemma \ref{lem_conditions} 
is satisfied. 
\end{lemma}
\begin{proof}
We have only to look at the coefficient of $s^n$ 
in the expansion of $\fn{B}{i}\fn{C}{i}$, 
which is 
\[
\z{i}{1}\f{B}{i}(\z{i}{1})\fnl{C}{i}{1}{n}(\z{i}{1})+
\z{i}{2}\f{C}{i}(\z{i}{2})\fnl{B}{i}{2}{n}(\z{i}{2})+
E_0+\z{i}{1}E_1(\z{i}{1})+\z{i}{2}E_2(\z{i}{2}), 
\]
where $E_0\in\CC$, $E_1(\z{i}{1})\in\cO_{X_i}(V\cap X_i)$ and 
$E_2(\z{i}{2})\in\cO_{X_0}((U_i\cap X_0)\setminus Q)$ 
are independent of $\fnl{B}{i}{2}{n}(\z{i}{2})$ and $\fnl{C}{i}{1}{n}(\z{i}{1})$. 
Since $\f{B}{i}(\z{i}{1})$ and $\f{C}{i}(\z{i}{2})$ 
are nowhere vanishing functions on 
$V\cap X_i$ and $(U_i\cap X_0)\setminus Q$, 
there are unique $\fnl{B}{i}{2}{n}(\z{i}{2})$ and $\fnl{C}{i}{1}{n}(\z{i}{1})$ 
such that the above expression belongs to $\CC$. 
\end{proof}

\begin{lemma}\label{lem_marked_points}
If $n=1$, 
the marked sections $\cP_i$ are mapped to $H_i$ 
if and only if $c_i(\z{i}{1})$ 
is proportional to $1-\z{i}{1}$. 
\end{lemma}
\begin{proof}
The condition here says that 
the images of $1$ and $\infty$ do not move. 
\end{proof}

\subsection{Tangent space}

Let us consider the case $n=1$. 

\begin{lemma}\label{lem_nodef_if_muiszero}
There are no nontrivial first order deformation in $\tilde{\fM}''$ 
with $\m{1}{1}=\m{2}{1}=0$. 
\end{lemma}
\begin{proof}
From Lemma \ref{lem_w0w1_on_X1X2}, 
we have 
\begin{eqnarray}
\fnl{A}{i}{0}{1} & = & c_i(0)\f{A}{i}(0), \label{eqn_A_i01}\\
\m{i}{1}\fnl{B}{i}{2}{0}(0) & = & d_ic_i(0)\f{B}{i}(0). \label{eqn_a_i1_B_i20}
\end{eqnarray}
If $\m{1}{1}=\m{2}{1}=0$, 
then $c_i(0)=0$, 
and Lemma \ref{lem_marked_points} shows that $c(\z{i}{1})=0$. 
Thus $\fnl{A}{i}{0}{1}+\z{i}{1}\fnl{A}{i}{1}{1}(\z{i}{1})=
\fnl{B}{i}{0}{1}+\z{i}{1}\fnl{B}{i}{1}{1}(\z{i}{1})=0$. 
The functions 
$\fnl{A}{i}{2}{1}(\z{i}{2})$ are uniquely determined 
by Lemma \ref{lem_w0_on_V0}. 
The constants $\fnl{C}{i}{0}{1}$ and functions $\fnl{C}{i}{2}{1}(\z{i}{2})$ 
are determined up to one dimensional parameter 
by Lemma \ref{lem_w2_on_V0}, 
and the condition $\tilde{f}(\cQ'(S))\subseteq H'$ makes them unique. 
Finally, 
$\fnl{B}{i}{2}{1}(\z{i}{2})$ and $\fnl{C}{i}{1}{1}(\z{i}{1})$ are unique 
by Lemma \ref{lem_predeformability}. 
\end{proof}

\begin{corollary}
The natural morphism $\tilde{\fM}''\to\cM$ is a closed immersion. 
\end{corollary}
\begin{proof}
The previous lemma says that the fiber over $(\m{1}{1}, \m{2}{1})=(0, 0)$ 
is a reduced point. 
Since $\tilde{\fM}''$ is the spectrum of a local ring which is finite dimensional 
as a $\CC$-vector space, 
the morphism $\tilde{\fM}''\to\cM$ is a closed immersion. 
\end{proof}

So we assume $(\m{1}{1}, \m{2}{1})\not=(0, 0)$ from now on. 

\begin{lemma}
(1)
If the conditions of Lemma \ref{lem_conditions} are satisfied with $n=1$ 
and $(\m{1}{1}, \m{2}{1})\not=(0, 0)$, 
then $\m{1}{1}: \m{2}{1} =  d_1\f{A}{2}'(0) : d_2\f{A}{1}'(0)$. 

(2)
The tangent space to $\tilde{\fM}''$ is at most $1$-dimensional. 
Thus $\tilde{\fM}''$ is isomorphic to $S_n$ for some $n$. 
\end{lemma}
\begin{proof}
(1)
By the equality (\ref{eqn_a_i1_B_i20}) in the proof of the previous lemma 
and Lemma \ref{lem_predeformability_0}, 
we see that $(c_1(0), c_2(0))\not=(0, 0)$. 
Using Lemma \ref{lem_predeformability_0} and Lemma \ref{lem_w0_on_V0}, 
we have 
\begin{eqnarray*}
\m{1}{1}: \m{2}{1} & = & 
\frac{d_1c_1(0)\f{B}{1}(0)}{\fnl{B}{1}{2}{0}(0)} : 
\frac{d_2c_2(0)\f{B}{2}(0)}{\fnl{B}{2}{2}{0}(0)} \\
 & = & \frac{d_1\f{C}{1}(0)}{\f{A}{1}'(0)\f{C}{1}'(0)} : 
\frac{d_2\f{C}{2}(0)}{\f{A}{2}'(0)\f{C}{2}'(0)}. 
\end{eqnarray*}
From the concrete form of $\f{C}{i}(\z{i}{2})$, 
this is equal to $d_1\f{A}{2}'(0) : d_2\f{A}{1}'(0)$. 

(2)
This follows from (1) and the previous corollary. 
\end{proof}

\begin{lemma}
If the conditions of Lemma \ref{lem_conditions} are satisfied 
and $(\m{1}{1}, \m{2}{1})\not=(0, 0)$, 
then $n<d_1$. 

In particular, $\tilde{\fM}''$ is reduced if $d_1=1$. 
\end{lemma}
\begin{proof}
Assume that $n\geq d_1$. 

If $d_1=d_2$, Lemma \ref{lem_mapped_into_Z1} gives 
$(\m{1}{1})^{d_1}\f{B}{1}(0)=(\m{2}{1})^{d_1}\f{B}{2}(0)$, 
and the previous lemma gives 
\[
(\f{A}{2}'(0))^{d_1}\f{B}{1}(0)=(\f{A}{1}'(0))^{d_1}\f{B}{2}(0). 
\]
This implies that the intersection multiplicity of 
$C_1$ and $C_2$ at $P$ is greater than $d_1$, 
which contradicts the assumption. 

If $d_1<d_2$, Lemma \ref{lem_mapped_into_Z1} says $\m{1}{1}=0$, 
which contradicts the previous lemma. 
\end{proof}

\begin{lemma}
If $d_1>1$, 
there exists a non-trivial first order deformation of $f$. 
\end{lemma}
\begin{proof}
We deform $(f, P_1, \dots, P_4)$ in $\tilde{\fM}$. 

Let $(\m{1}{1}, \m{2}{1}) =  (d_1\f{A}{2}'(0), d_2\f{A}{1}'(0))$, 
and take any $c_i(\z{i}{1})$ satisfying 
equality (\ref{eqn_a_i1_B_i20})
in the proof of Lemma \ref{lem_nodef_if_muiszero}. 
Then one can find 
$\fnl{A}{i}{0}{1}$, $\fnl{A}{i}{1}{1}(\z{i}{1})$, 
$\fnl{B}{i}{0}{1}$ and $\fnl{B}{i}{1}{1}(\z{i}{1})$ 
satisfying the conditions of Lemma \ref{lem_w0w1_on_X1X2}(1). 
By our choice of $(\m{1}{1}, \m{2}{1})$, 
we have $\fnl{A}{1}{0}{1}=\fnl{A}{2}{0}{1}$, 
and one can find $\fnl{A}{i}{2}{1}(\z{i}{2})$ 
as in Lemma \ref{lem_w0_on_V0}. 
Choose $\fnl{C}{i}{0}{1}$ and $\fnl{C}{i}{2}{1}(\z{i}{2})$ 
as in Lemma \ref{lem_w2_on_V0}. 

Then conditions (\ref{w0_on_V0}), (\ref{w2_on_V0}), 
(\ref{w0w1_on_X1X2}) 
of Lemma \ref{lem_conditions} are satisfied. 
By Lemma \ref{lem_w1_on_V0_w2_on_X1X2} 
and Lemma \ref{lem_mapped_into_Z1}, 
conditions (\ref{w1_on_V0_w2_on_X1X2}) and 
(\ref{mapped_into_Z1}) are always satisfied. 

Finally, $\fnl{B}{i}{2}{1}(\z{i}{2})$ and $\fnl{C}{i}{1}{1}(\z{i}{1})$ 
can be chosen to satisfy (\ref{predeformability}) 
by Lemma \ref{lem_predeformability}. 

Since $\cX$ is a non-trivial family over $S_1$, 
the induced morphism $S_1\to\fM$ is non-constant. 
\end{proof}

\subsection{Conclusion}

For $d_1>1$, 
we will show that  $\fM$ 
is isomorphic to $S_{d_1-1}$. 
It is sufficient to show that 
there is a family over $S_{d_1-1}$ belonging to $\tilde{\fM}$ 
such that the tangent map of the induced morphism $S_{d_1-1}\to\fM$ 
is nonzero. 
We have already shown that there is a nontrivial first order deformation, 
so the problem is reduced to the following lemma. 

\begin{lemma}
If $1<n < d_1$ and a family in $\tilde{\fM}$ over $S_{n-1}$ is given, 
then it can be extended to a family over $S_n$. 
\end{lemma}
\begin{proof}
Similar to the proof of the previous lemma. 
Just note that 
$\fnl{A}{1}{0}{n}$ and $\fnl{A}{2}{0}{n}$ can be 
arbitrarily chosen 
by Lemma \ref{lem_w0w1_on_X1X2}(2) 
and so 
Lemma \ref{lem_w0_on_V0} can be applied. 
\end{proof}

\section{Rational quartics with full tangency to a cubic}

As a check on our formula, 
let us look at 
rational plane curves of low degrees 
with full tangency to a smooth cubic. 

Let $Z$ be the projective plane $\PP^2$ 
and $D$ a smooth cubic. 
The genus $0$, degree $d$, full tangency 
relative Gromov-Witten invariants $I_d$ of $(Z, D)$ 
were calculated in \cite{G}. 
In small degrees, the invariants are 
$I_1=9$, $I_2=135/4$, $I_3=244$, 
$I_4=36999/16$.

Now let us describe the set of image curves 
of relative stable morphisms. 
Let $H$ denote a line in $\PP^2$. 

\begin{definition}
For a positive integer $d$, 
let $\bar{M}_d$ denote the image 
of the natural map 
from $\fM_{d[H]}$ to the Chow variety of $Z$. 
For $P\in D$, let 
\[
\bar{M}_{d, P}=\{C\in \bar{M}_d| C\cap D=\{P\}\}. 
\]
Denote by $M_d$ and $M_{d, P}$ 
the subsets of $\bar{M}_d$ and $\bar{M}_{d, P}$ 
consisting of irreducible and reduced curves. 
\end{definition}

In the following, 
let us denote the normalization of a variety $X$ by $X^\nu$. 
We choose a flex $O\in D$ and regard $D$ as a group 
with the zero element $O$. 

\begin{proposition}
(1)
The set $\bar{M}_d$ is the union of $\bar{M}_{d, P}$ 
where $P$ runs $3d$-torsion points of $D$. 

(2)
We have 
$\bar{M}_{d, P} = \{C_1+\dots+C_r| r\geq 1, d_1+\dots+d_r=d, C_i\in M_{d_i, P}\}$. 

(3)
An effective divisor $C$ of degree $d$ is contained in $M_d$ 
if and only if it is irreducible and reduced 
and $(C\setminus D)^\nu\cong\bA^1$. 
The normalization map $\PP^1\to Z$ 
gives the unique point in $\fM_{d[H]}$ over $C\in M_d$. 
\end{proposition}
\begin{proof}
(1)
Let $C$ be an element of $\bar{M}_d$. 
By Lemma \ref{lem_configuration}(3), 
$({\rm Supp\ } C)\cap D$ consists of one point $P$. 
From $C|_D\sim dH|_D$, 
it follows that $P$ is a $3d$-torsion. 

(2)
Let $C$ be an element of $\bar{M}_d$. 
By Corollary \ref{cor_configuration}(1), 
we have morphisms $f_i: \PP^1\to Z$ 
such that $f_i^{-1}D$ consists of one point 
and $C=\sum e_iC_i$, 
where $C_i=f_i(\PP^1)$ and $e_i=\deg f_i$. 
Let $\nu_i: Y_i\to C_i$ be the normalization. 
Then $f_i$ factors through $Y_i$ 
and so $\nu_i^{-1}D$ consists of one point. 
Thus $\nu_i$ belongs to $\fM_{[C_i]}$ 
and $C_i$ is an element of $M_{\deg C_i}$. 

(3)
If $C\in M_d$ comes from 
a relative stable morphism $f$, 
then by Corollary \ref{cor_configuration}(2) 
the domain of $f$ is isomorphic to $\PP^1$, 
the target is $Z$ and $f$ is a normalization morphism for $C$. 
Since $f^{-1}(D)$ consists of one point, 
we have $(C\setminus D)^\nu\cong\bA^1$. 

Conversely, 
assume that $C$ is an irreducible and reduced curve of degree $d$ 
such that $(C\setminus D)^\nu\cong\bA^1$. 
Let $\nu: C^\nu\to C$ be the normalization map. 
Since $\nu^{-1}(C\setminus D)$ is isomorphic to $\bA^1$, 
$C^\nu$ is isomorphic to $\PP^1$ and $\nu^{-1}(D)$ 
consists of one point. 
Thus $\nu$ belongs to $\fM_{d[H]}$. 
\end{proof}

\begin{proposition}
(1)
If $P$ is a flex, 
$\bar{M}_{1, P}=M_{1, P}$ consists of a line. 

(2)
If $P$ is a flex, 
$\bar{M}_{2, P}$ consists of a double line. 
If $P$ is a $6$-torsion which is not a flex, 
$\bar{M}_{2, P}=M_{2, P}$ consists of a smooth conic. 

(3)
If $P$ is a flex, 
$\bar{M}_{3, P}$ consists of a triple line and $2$ nodal cubics (resp. $1$ cuspidal cubic) 
if $D\not\cong(y^2=x^3-1)$ (resp. $D\cong(y^2=x^3-1)$). 
If $P$ is a $9$-torsion which is not a flex, 
$\bar{M}_{3, P}=M_{3, P}$ consists of $3$ nodal cubics. 
\end{proposition}

\begin{proof}(Outline) 
By the previous proposition, 
we have only to know $M_d$. 

If the degree $d$ is $1$ or $2$, 
there is a unique effective divisor $C$ such that $C|_D=3dP$, 
and the assertions are easy to see. 

If $d=3$, the set 
$\Lambda=\{\hbox{$C$: effective divisor with $C|_D=9P$}\}\cup\{D\}$ is a linear pencil, 
and we obtain a rational elliptic surface $f: F\to \PP^1$ 
by resolving the base points. 
There is a unique member $C_0$ of $\Lambda$ which is singular at $P$. 
If $P$ is a flex (resp. $P$ is not a flex), $C_0$ is a triple line 
(resp. an irreducible cubic with a node at $P$). 
The corresponding fiber of $f$ has Euler characteristic $10$ (resp. $9$), 
and it follows that $M_{3, P}$ has two nodal cubics or one cuspidal cubic 
(resp. three nodal cubics or 
one nodal cubic and one cuspidal cubic). 
A little more detailed analysis shows the assertions. 
See \cite[Propositions 1.4 and 1.5]{T} . 
\end{proof}

There are $9$ flexes on $D$, 
and so $\bar{M}_1=M_1$ consists of $9$ lines. 
This coincides with $I_1=9$. 

In degree $2$, 
there is a double line at each flex $P$, contributing $M_3[2]=3/4$. 
There are $27$ points $P$ with $3P\sim 3O$ and $6P\not\sim 6O$, 
and so $M_2$ consists of $27$ conics. 
They sum up to $9\times 3/4+27=135/4=I_2$. 

For cubics, 
we similarly have 
\[
9\times (M_3[3] + 2) + 72\times 3 = 244 = I_3. 
\]

\begin{proposition}
Let $D$ be a general cubic. 

(1)
If $P$ is a flex, 
$\bar{M}_{4, P}$ consists of 
\begin{itemize}
\item
$1$ quadruple line, 
\item
$2$ reducible curves of the form $L+C$ 
where $L$ is the tangent line at $P$ 
and $C$ is an element of $M_{3, P}$, 
and 
\item
eight immersed rational curves in $M_{4, P}$. 
\end{itemize}

(2)
If $P$ is a $6$-torsion which is not a flex, 
$\bar{M}_{4, P}$ consists of 
\begin{itemize}
\item
$1$ double conic and 
\item
$14$ immersed rational curves in $M_{4, P}$. 
\end{itemize}

(3)
If $P$ is a $12$-torsion which is not a $6$-torsion, 
$\bar{M}_{4, P}=M_{4, P}$ consists of 
$16$ immersed rational curves. 
\end{proposition}

\begin{proof}
Let 
\begin{eqnarray*}
T_1 & = & \{P\in D| 3(P-O)\sim 0\}, \\ 
T_2 & = & \{P\in D| 3(P-O)\not\sim 0, 6(P-O)\sim 0\} \hbox{ and}\\ 
T_3 & = & \{P\in D| 6(P-O)\not\sim 0, 12(P-O)\sim 0\}. 
\end{eqnarray*}
We have only to show that $M_{4, P}$ contains 
$8$, $14$ or $16$ curves if $P\in T_1, T_2$ or $T_3$ respectively, 
and that they are immersed. 

We use the following cover to count the curves. 
\begin{lemma}
Let $\pi: Y\to Z$ be the triple cover 
totally ramified over $D$ and unramified elsewhere, 
and let $D_Y$ denote the divisor $(\pi^{-1}D)_{\rm red}$. 
Let $g: Y\to Z'$ be a blow-down to $Z'\cong \PP^2$, 
$D'$ the image $g(D_Y)$, 
$E_i$ ($i=1, \dots, 6$) the exceptional curves 
and $P_i=g(E_i)$. 

(1)
$Y$ is a smooth cubic surface 
and $D_Y$ is a smooth plane section. 
The curves $D$, $D_Y$ and $D'$ are isomorphic 
via the natural maps. 
(So we identify points on $D$, $D_Y$ and $D'$.)

(2)
For any positive integer $d$, 
let 
\[
\tilde{M}_{d, P}=
\{C\subset Y| (C\setminus D_Y)^\nu\cong\bA^1, \deg C=d, C\cap D_Y=P\}. 
\]
Then $C\mapsto \pi(C)$ defines 
a $3$-to-$1$ map $\tilde{M}_{d, P}\to M_{d, P}$, 
and the morphism $C\to \pi(C)$ is birational and unramified outside $D_Y$. 

(3)
The exceptional curves $E_i$ are lines 
and $\pi$ maps them isomorphically 
to flex tangent lines. 

Replacing $O$ by another flex and renumberin $P_i$ if necessary, 
there is an isomorphism $\theta: D\cong (\RR/\ZZ)^2$ of groups 
such that 
\begin{eqnarray*}
& & \theta(P_1) = (0, 0), \theta(P_2)=(1/3, 0), \theta(P_3)=(2/3, 0) \\
& & \theta(P_4) = (0, 1/3), \theta(P_5)=(1/3, 1/3), \theta(P_6)=(2/3, 1/3)
\end{eqnarray*}
and $\theta(O')=(1/9, 0)$ for a flex $O'\in D'$. 
\end{lemma}

\begin{proof}
(1)
If $D$ is defined by a homogeneous cubic polynomial 
$F(x, y, z)=0$, 
then $Y$ can be described as $F(x, y, z)=w^3$ 
and $D_Y$ as $w=0$. 
It is easy to see that $Y$ is smooth. 

Since $\pi$ is totally unramified over $D$, 
the morphism $D_Y\to D$ is an isomorphism. 
It is well known that 
a smooth plane section of a smooth cubic surface 
is mapped isomorphically to a smooth plane cubic 
by a blow-down to the projective plane. 

(2)
Let $C_1$ be an element of $\tilde{M}_{d, P}$ 
and $C=\pi(C_1)$. 
Then $C\to C_1$ is unramified outside $D_Y$, 
and it follows that $\bA^1\cong(C_1\setminus D_Y)^\nu\to (C\setminus D)^\nu$ 
is \'etale. 
Therefore it is an isomorphism, 
and $C$ belongs to $M_{d, P}$. 

Conversely, 
let $C$ be an element of $M_{d, P}$ 
and $\tilde{C}=\pi^{-1}(C)$. 
The morphism $\tilde{C}^\nu\to C^\nu$ is of degree $3$ 
and unramified except over one point. 
Thus it is a trivial $3$-sheeted cover, 
and $\tilde{C}$ is decomposed into $3$ components $C_1$, $C_2$ and $C_3$, 
each mapping birationally to $C$. 
Since $C_i$ are exchanged by the action of the Galois group, 
we have $\deg C_i=d$. 
Since $(C_i\setminus D_Y)^\nu$ maps isomorphically to $(C\setminus D)^\nu$, 
$C_i$ belongs to $\tilde{M}_{d, P}$. 

(3)
There are $9$ flex tangent lines to $D$ 
and each of them gives $3$ lines on $Y$. 
Thus all the $27$ lines, hence $E_i$'s, are obtained in this way. 

Let $O=P_1$ and 
choose a flex $O'$ of $D'$. 
Since $E_i$ are disjoint, $P_i$ are all different. 
They are flexes on $D$ 
and so there are $3$ remaining flexes, 
which we denote by $Q_1, Q_2$ and $Q_3$. 

Let $C_1, C_2$ and $C_3$ be lines on $Y$ through $Q_1$ 
and let $L_j=g(C_j)$. 
They are lines 
since they have positive degrees and 
the sum of degrees is $3$. 
Thus $L_j|_{D'}=Q_1+P_{a_j}+P_{b_j}$ 
for some $a_j\not=b_j$. 
We have $3(P_{a_j}+P_{b_j}-2O)\sim 0$ 
since $P_i$ are $3$-torsions, 
and the class of $P_{a_j}+P_{b_j}\sim H|_{D'}-Q_1$ 
is independent of $j$. 
Thus we have $\sum_{i=1}^3(P_{a_j}+P_{b_j})\sim 6O$. 
We also have $\{a_j, b_j\}\cap\{a_{j'}, b_{j'}\}=\emptyset$ if $j\not=j'$ 
since $L_j$ and $L_{j'}$ pass through $Q_1$. 
It follows that $\sum_{i=1}^6 P_i\sim 6O$ 
and $Q_1+Q_2+Q_3\sim 3O$. 
Therefore we can find $\theta: D\cong (\RR/\ZZ)^2$ 
such that 
$\{\theta(Q_1), \theta(Q_2), \theta(Q_3)\}=\{(p/3, 2/3)| p=0, 1, 2\}$. 
We can renumber $P_i$ as in the assertion. 

From $Q_1+P_{a_j}+P_{b_j}\sim 3O'$ 
it follows that $O'$ is a $9$-torsion. 
If it were a $3$-torsion, 
we would have $P_1+P_2+P_3\sim 3O'$. 
Thus $P_1$, $P_2$ and $P_3$ lie on a line in $Z'$, 
which contradicts the fact that $D_Y$ is ample. 
Consequently, $O'$ is of order $9$. 

Let $\theta_2$ denote the composition of $\theta$ 
and the second projection. 
From the fact that $\theta(P_{a_j})+\theta(P_{b_j})$ is 
independent of $j$, 
it follows that $\theta_2(P_{a_j})+\theta_2(P_{b_j})=1/3$. 
We know that $\theta_2(Q_1)=2/3$, 
so $\theta_2(O')$ is a $3$-torsion. 
By replacing $\theta$ and $O'$, 
we may assume that $\theta(O')=(1/9, 0)$. 
\end{proof}

\begin{lemma}\label{lem_divisor_class}
Let $C$ be an irreducible and reduced curve of degree $4$ on $Y$ 
and write $C\sim e\cdot g^*H-\sum a_iE_i$. 
Write $[a]$ for the unordered sequence $[a_1, \dots, a_6]$. 
Then the pair $(e, [a])$ 
is contained in the following list. 
The column ``$p_a$'' gives the arithmetic genera 
and ``$\#$'' the numbers of ordered sequences $(a_1, \dots, a_6)$. 
\begin{multicols}{3}
\begin{tabular}{|c|c|c|c|}
\hline
$e$ & $[a]$ & $p_a$ & $\#$ \\
\hline
$2$ & $[0,0,0,0,1,1]$ & $0$ & $15$ \\
\hline
$3$ & $[0,0,1,1,1,2]$ & $0$ & $60$ \\
 & $[0,1,1,1,1,1]$ & $1$ & $6$ \\
\hline
\end{tabular}
\columnbreak

\begin{tabular}{|c|c|c|c|}
\hline
$e$ & $[a]$ & $p_a$ & $\#$ \\
\hline
$4$ & $[0,1,1,2,2,2]$ & $0$ & $60$ \\
 & $[1,1,1,1,1,3]$ & $0$ & $6$ \\
 & $[1,1,1,1,2,2]$ & $1$ & $15$ \\
\hline
\end{tabular}
\columnbreak

\begin{tabular}{|c|c|c|c|}
\hline
$e$ & $[a]$ & $p_a$ & $\#$ \\
\hline
$5$ & $[1,1,2,2,2,3]$ & $0$ & $60$ \\
 & $[1,2,2,2,2,2]$ & $1$ & $6$ \\
\hline
$6$ & $[2,2,2,2,3,3]$ & $0$ & $15$ \\
\hline
\end{tabular}
\end{multicols}

If $p_a=0$ (resp. $p_a=1$), 
there exists a blowdown $g': Y\to \PP^2$ 
of lines $E'_1, \dots, E'_6$ 
such that $C\sim 2(g')^*H-E'_{i_1}-E'_{i_2}$ 
(resp. $C\sim 3(g')^*H - E'_{i_1}-\dots-E'_{i_5}$). 
\end{lemma}

\begin{proof}
Since $C$ is not a line, 
we have $0\leq a_i=C.E_i\leq C.\pi^*\pi_*E_i=4$. 
We have 
\begin{equation}\label{eqn_3e}
3e=\pi_*C.D'=C.(D_Y+\sum E_i)=4+\sum a_i 
\end{equation}
and in particular $2\leq e\leq 9$. 
Solving (e.g. by an exhaustive search) 
\[
p_a(C) = \frac{(e-1)(e-2)}{2} - \sum \frac{a_i(a_i-1)}{2} \geq 0 
\]
under the conditions $0\leq a_i\leq 4$, $2\leq e\leq 9$ and (\ref{eqn_3e}), 
we obtain the list. 

If $p_a=0$ and $e>2$ or $p_a=1$ and $e>3$, 
one can find pairwise distinct $i, j$ and $k$ such that $a_i+a_j+a_k>e$. 
By replacing $Z'$ by its quadratic transform 
with fundamental points $P_i, P_j$ and $P_k$, 
we can make $e$ smaller, 
hence the second assertion. 
\end{proof}

\begin{lemma}\label{lem_immersed}
If $D$ is a general cubic and $C$ is an element of $M_{4, P}$, 
then $C$ is immersed. 
\end{lemma}
\begin{proof}
Since $P\in C$ has only one analytic branch, 
it must be either a smooth point, an ordinary cusp 
or a cusp of type $\eta^3=\xi^4$. 
From $(C.D)_P=12$, 
it follows that $C$ is smooth at $P$. 

If $D'$ is another smooth cubic such that $(C.D')_P=12$, 
then $(D.D')_P\geq 12$ and so $D=D'$. 
Thus it suffices to show that 
there are finite possibilities (up to projective equivalence) 
for non-immersed quartics $C$ and smooth points $P\in C$ such that 
there exists a cubic $D$ smooth at $P$ with $(C.D)_P=12$. 

We have $C=\pi(C_1)$ for some $C_1\in \tilde{M}_{4, P}$ 
and $C_1$ is mapped to a conic or a cubic by a blowdown $Y\to\PP^2$. 
Since $\pi|_{C_1}$ is unramified outside $D$, 
the possibilities are as follows. 
\begin{enumerate}
\item\label{two_nodes}
$C$ has $2$ nodes and $1$ cusp. 
\item\label{tacnode}
$C$ has $1$ tacnode and $1$ cusp. 
\item\label{triple_point}
$C$ has a singular point with one smooth branch 
and one cuspidal branch. 
\end{enumerate}

Let $\nu: \PP^1\to C$ be the normalization. 
Let $(s: t)$ be a homogeneous coordinate on $\PP^1$, 
and for $x\in\CC\cup\{\infty\}$ 
denote by $[x]$ the point $(s: t)$ with $t/s=x$. 
For $a, b\in\CC\cup\{\infty\}$, 
write $P_{a, b}$ for the nodal rational curve 
obtained by gluing $[a]$ and $[b]$, 
$P_a$ for the curve obtained by making $[a]$ a cusp. 
Write $[x]$ also for the image of $[x]\in\PP^1$ 
in $P_a$ etc.

In case (\ref{two_nodes}), 
write $Q$ and $R$ for the nodes and $S$ for the cusp. 
We may assume that 
$\nu^{-1}(Q)=\{0, 1\}$, 
$\nu^{-1}(R)=\{a, b\}$ and 
$\nu^{-1}(S)=\{\infty\}$. 
Let $p$ correspond to $P$. 
Denote by $L_0$, $L_1$ and $L_2$ 
the lines $RS$, $SQ$ and $QR$. 
Considering the pullbacks of $L_i$, 
we see that $\nu$ is determined to be 
$(s^2(t-as)(t-bs): s^2t(t-s): t(t-1)(t-as)(t-bs))$ 
up to projective equivalence. 

From $D\sim 3L_2$, 
we have $12[p]\sim 3[0]+3[1]+3[a]+3[b]$ on $P_\infty$. 
Similarly, 
we have $12[p]\sim 3[a]+3[b]+6[\infty]$ on $P_{0, 1}$ 
fand 
$12[p]\sim 3[0]+3[1]+6[\infty]$ on $P_{a, b}$. 
Hence we have
\begin{eqnarray*}
4p & = & a+b+1, \\
\left(\frac{p-1}{p}\right)^{12} & = & \left(\frac{(a-1)(b-1)}{ab}\right)^3, \\
\left(\frac{p-b}{p-a}\right)^{12} & = & \left(\frac{b(b-1)}{a(a-1)}\right)^3. 
\end{eqnarray*}

Eliminating $p$, clearing fractions 
and calculating the greatest common divisor of the resulting 
polynomials, 
we see that the solutions are finite outside 
\[
\left\{(p, a, b)\left| 
4p=a+b+1, 
\left(\frac{p-1}{p}\right)^4=\frac{(a-1)(b-1)}{ab}, 
\left(\frac{p-b}{p-a}\right)^4=\frac{b(b-1)}{a(a-1)}
\right.\right\}. 
\]
Take $(p, a, b)$ from this set. 
We assume that 
it gives a curve $C$ satisfying the conditions 
and draw a contradiction. 
Let $H$ be a general line 
and $\mathfrak{h}=H|_C$. 
Then $4[p]$ is linearly equivalent to the pullbacks of $\mathfrak{h}$ 
on $P_\infty$, $P_{0, 1}$ and $P_{a, b}$. 
Thus there are rational functions $f_1, f_2$ and $f_3$ 
which are regular and nonzero at 
$Q$, $R$ and $S$, respectively, 
such that $(f_i)=4[p]-\nu^*\mathfrak{h}$ on $\PP^1$. 
They differ only by nonzero constant factors, 
so we conclude that $4[p]\sim\mathfrak{h}$. 
Since $H^0(\cO_Z(H))\to H^0(\cO_C(\mathfrak{h}))$ is surjective, 
there is a line $L$ with $L|_C=4P$. 
Then we would have $(L.D)_P\geq 4$, a contradiction. 

In case (\ref{tacnode}), 
we may assume that 
$\nu^{-1}(\hbox{tacnode})=\{0, 1\}$ and 
$\nu^{-1}(\hbox{cusp})=\{\infty\}$. 
Let $L_0$ be the tangent line at the cusp, 
$L_1$ the line joining the cusp and the tacnode 
and $L_2$ the tangent line at the tacnode. 
Let $a\in\PP^1$ correspond to the fourth intersection 
of $L_0$ and $C$, 
and $p$ to $P$. 
Then $\nu$ is determined to be 
$(s^3(t-as): s^2t(t-s): t^2(t-s)^2)$ up to projective equivalence. 

From $D\sim 3L_2$, 
we have $12[p]\sim 6[0]+6[1]$ on $P_\infty$, 
hence $12p=6$. 
From $D\sim 3L_0$, 
we have $12[p]\sim 3[a]+9[\infty]$ on $P_{0, 1}$, 
hence $((p-1)/p)^{12}=((a-1)/a)^3$. 
So there are finite possibilities for $a$ and $p$. 

In case (\ref{triple_point}), 
we may assume that 
$0$ and $\infty$ correspond to 
the smooth and cuspidal branches, respectively. 
Let $L_0$ be the tangent line to the cuspidal branch 
and $L_1$ the tangent line to the smooth branch. 

The dual curve of any quartic curve 
must be singular, since the dual of a smooth curve 
cannot be of degree $4$. 
This means that $C$ has a tangent line $L_2$ 
such that $\nu^*L_2=k[a]+(4-k)[b]$ with $k\leq 2$ and $a\not=b$. 
It is clear that $L_2$ does not pass through the singular point, 
and so we may assume that $b=1$ and $a\not=0, 1, \infty$. 
Now $\nu$ is determined to be 
$(s^3t: s^2t^2: (t-as)^k(t-s)^{4-k})$ up to projective equivalence 
(and $a$ is irrelevant if $k=0$). 

From $D\sim 3L_2$, 
we have $12[p]\sim 3k[a]+3(4-k)[1]$ on $P_\infty$ and $P_{0, \infty}$, 
hence $12p=3ka+3(4-k)$ and $p^{12}=a^{3k}$. 
So there are finite possibilities for $p$, 
and $a$ is determined by $p$ if $k\not=0$. 
\end{proof}

\begin{lemma}
Let $D$ be general. 
Then $\#M_{4, P}$ depends only on 
the index $i$ such that $P$ belongs to $T_i$. 
\end{lemma}
\begin{proof}
In fact, if $P$ and $P'$ belong to the same $T_i$, 
they can be related by projective equivalence and 
monodromy. 
\end{proof}

Thus we have only to calculate 
$N_i=\sum_{P\in T_i}\#\tilde{M}_{4, P}$

\begin{lemma}
Let $A=e\cdot g^*H-\sum a_iE_i$, 
where $(e, [a])$ is in the list of Lemma \ref{lem_divisor_class}. 

(1)
If $p_a(A)=0$, 
$|A|$ has $1$ curve contributing to $N_1$, 
$3$ curves contributing to $N_2$ 
and $12$ curves contributing to $N_3$. 

(2)
If $p_a(A)=1$, 
$|A|$ has no curve contributing to $N_1$, 
$6\times 3$ curves contributing to $N_2$ 
and $8\times 12$ curves contributing to $N_3$. 
\end{lemma}
\begin{proof}
It suffices to consider the cases $(e, p_a)=(2, 0)$ and $(e, p_a)=(3, 1)$. 

(1)
Write $A=2g^*H-E_i-E_j$ 
and let $T=\{P; 4P\sim A|_{D_Y}\}$. 
An element of $T$ is a $12$-torsion 
since $A|_{D_Y}\sim 6O'-P_i-P_j$, 
and one can find a $3$-torsion in $T$ 
by adding a $4$-torsion. 
Thus $T$ contains $1$ element of $T_1$, 
$3$ elements of $T_2$ and $12$ elements of $T_3$. 

Let $P$ be any element of $T$. 
Then there is a unique effective divisor $C\subset Z'$ of degree $2$ 
such that $C|_{D'}=P_i+P_j+4P$. 
If $C$ were reducible or non-reduced, 
we would have $3P\sim P_i+P_j+P\sim 3O'$ 
or $P_i+2P\sim P_j+2P\sim 3O'$. 
In the former case, one has 
$P\sim 3O'-P_i-P_j$ and 
$3O'\sim 3P\sim 9O'-3P_i-3P_i\sim 3O$, 
which is a contradiction. 
The latter is also impossible because $P_i\not=P_j$. 
Thus $C$ is irreducible and reduced, 
and contributes to $\tilde{M}_{4, P}$. 

(2)
Write $A=3g^*H-E_1-\dots-\check{E}_i-\dots-E_6$ 
and let $T=\{P; 4P\sim A|_{D_Y}\}$. 
Then $T$ contains $P_i$, 
and so $T$ is the set of points $P$ 
such that $P-P_i$ is a $4$-torsion. 

Let $\Lambda=\{C\subset Z'; C|_{D'}=4P+P_1+\dots+\check{P}_i+\dots+P_6\}\cup\{D'\}$. 
This is a linear pencil. 
If $C\in |A|$, 
then $g_*C\in\Lambda$ 
and $C=g^*g_*C-E_1-\dots-\check{E}_i-\dots-E_6$. 

If $P=P_i$, 
then any $C\in\Lambda$ passes through $P_i$ 
and the corresponding curve $g^*C-E_1-\dots-\check{E}_i-\dots-E_6\in|A|$ 
is reducible. 
Thus there is no contribution to $\tilde{M}_{4, P}$.

\begin{claim}
If $P-P_i$ is of order $2$, 
$\Lambda$ contains exactly one element 
which is not irreducible and reduced, 
and it is of the form $C_1+C_2$ 
with $C_1|_{D'}=2P+P_{i_1}$ 
and $C_2|_{D'}=2P+P_{i_2}+\dots+P_{i_5}$. 

If $P-P_i$ is of order $4$, 
$\Lambda$ contains 
no reducible or non-reduced element. 
\end{claim}
\begin{proof}
Assume that $C=C_1+C_2$ with $\deg C_i=i$. 
Since no three of $P_i$'s lie on a line in $Z'$, 
$k=(D'.C_1)_P$ is not zero, 
and $kP+\sum_{j=1}^{3-k}P_{i_j}\sim 3O'$ holds. 

Multiplying by $3$, 
we have $3k(P-P_i)\sim 0$, 
and therefore $k=2$ and $P-P_i$ is of order $2$. 
If $P-P_i$ is of order $2$, 
then we have $2P+P_{i_1}\sim 3O'$ 
for the index $i_1$ 
such that $P_{i_1}\sim P_i+P_2-O$. 
\end{proof}

There is a unique element $C_0$ of $\Lambda$ 
which is singular at $P$, 
since being singular at $P$ is a linear condition. 
If $P-P_i$ is of order $2$, 
this is $C_1+C_2$ above. 
If $P-P_i$ is of order $4$, 
then $C_0$ is irreducible, 
and from $(C_0.D')_P=4$ 
it follows that 
the singularity at $P$ is a node 
whose branches intersect $D'$ with multiplicities $1$ and $3$. 

Let $F$ be the surface obtained by 
blowing up $P_1, \dots, \check{P}_i, \dots, P_6$ 
and the $4$ consecutive infinitely near points on $D'$ 
over $P$. 
This is a resolution of base points of $\Lambda$, 
and we have an elliptic fibration $f: F\to \PP^1$. 
The Euler characteristic of the fiber corresponding to $C_0$ 
is $6$ or $4$ if the order of $P-P_i$ is $2$ or $4$, respectively. 

The other singular fibers of $f$ 
correspond to members of $\Lambda$ 
which are singular on $Z'\setminus D'$. 
They are nodal by Lemma \ref{lem_immersed}. 
Since the Euler characteristic of $F$ is $12$, 
there are $6$ or $8$ such curves 
if the order of $P-P_i$ is $2$ or $4$, respectively. 
They are smooth at $P_1, \dots, \check{P_i}, \dots, P_6$ 
and $P$, 
and so they give elements of $\tilde{M}_{4, P}$. 

Since there are $3$ points of order $2$ 
and $12$ points of order $4$, 
the assertion was proved. 
\end{proof}

The list has $216$ divisor classes with $p_a=0$ 
and $27$ classes with $p_a=1$, and so we have 
\begin{eqnarray*}
N_1 & = & 216 \times 1 + 27\times 0, \\
N_2 & = & 216 \times 3 \times 1 + 27 \times 3 \times 6, \hbox{ and} \\
N_3 & = & 216 \times 12 \times 1 + 27 \times 12 \times 8. 
\end{eqnarray*}
Since $\tilde{M}_{4, P}\to M_{4, P}$ is $3$-to-$1$ 
and curves are evenly distributed in $T_i$, 
we have 
$\# M_{4, P}=N_i/(3\times\# T_i)$, 
and it is calculated to be $8$, $14$ or $16$ 
if $i=1, 2$ or $3$. 
\end{proof}

By our main theorem, 
$(\hbox{line})+(\hbox{cubic})$ contributes $3$. 
Thus the relative Gromov-Witten invariant should be 
\[
9\times (M_3[4] + 2\times 3 + 8) 
+ 27\times (M_6[2] + 14)
+ 108\times 16
= \frac{36999}{4},
\]
which is surely equal to $I_4$.

\begin{remark}
Let $L$ denote the line bundle over $\PP^2$ 
associated to $\cO_{\PP^2}(-3)$, 
and let $K_n$ denote the genus $0$, degree $n$ 
local Gromov-Witten invariant of $L$. 
By \cite{CKYZ} and \cite{G}, 
$I_n = (-1)^{n-1}3n K_n$ holds. 
The multiple cover formula for $L$ 
suggests one to consider 
$M'_n[d]=(-1)^{n(d-1)}/d^2$ 
as the contribution of $d$-fold covers of 
degree $n$ curves. 

One can define $m_w[d]$ by 
$M_w[d]=\sum_{d=d_1d_2} M'_{d_1w}[d_2]m_w[d_1]$. 
(Note that $M'_n[d]=M'_{3n}[d]$.) 
For small values of $w$ and $d$
one observes that $m_w[d]$ are positive integers. 
We may view this in the following way: 
If $C$ is an immersed rational curve of degree $n$ with full tangency, 
it gives $m_{C.D}[d]$ ``instantons'' in the class $d[C]$. 
Then the contribution of ``$d$-fold covers'' of each instanton of class $\beta$ to 
the relative Gromov-Witten invariant is $M'_{D.\beta}[d]$. 

Remarkably, it seems that 
all $3d$-torsion points have the same number of instantons. 
For example, 
we have $1\times m_3[4] + 2\times 3 + 8=16$ instantons 
at $P\in T_1$ 
and $1\times m_6[2] + 14 = 16$ instantons 
at $P\in T_2$. 
\end{remark}

\end{document}